\documentclass{amsart}
\usepackage[T1]{fontenc}
\usepackage[latin9]{inputenc}
\usepackage{geometry}
\usepackage{amsmath,latexsym,amssymb,amsfonts}
\makeatletter

\newtheorem{theorem}{Theorem}
\newtheorem{lemma}[theorem]{Lemma}
\newtheorem{prop}[theorem]{Proposition}
\newtheorem{corollary}[theorem]{Corollary}
\newtheorem{definition}[theorem]{Definition}
\newtheorem{remark}[theorem]{Remark}
\newtheorem{example}[theorem]{Example}

\makeatother

\title{Free probability of type $B$: analytic interpretation and applications}
\author[S.T. Belinschi 
and 
D. Shlyakhtenko]{S.T. Belinschi$^\star$ 
and 
D. Shlyakhtenko$\ddagger$}
\date{
March 15, 2009\newline\indent
$^\star$ University of Saskatchewan and Institute of Mathematics of the Romanian Academy.
106 Wiggins Road, Saskatoon, Saskatchewan S7N 5E6, Canada.
E-mail: belinschi@math.usask.ca.
Supported by a Discovery grant from the Natural Sciences and
Engineering Research Council of Canada and a University of Saskatchewan
start-up grant.
\newline\indent
$^\ddagger$ Department of Mathematics, UCLA, Los Angeles, CA 94720, U.S.A. E-mail: shlyakht@math.ucla.edu. Research supported by NSF grant DMS-0555680}
\begin{document}

\maketitle

\begin{abstract}
In this paper we give an analytic interpretation of free
convolution of type $B$, and provide a new formula for its
computation. This formula allows us to show that free additive 
convolution of type $B$ is essentially a re-casting of conditionally
free convolution. We put in evidence several aspects of this operation,
the most significant being its apparition as an 'intertwiner'
between derivation and free convolution of type $A$. We also
show connections between several limit theorems in type $A$
and type $B$ free probability.  Moreover, we show that the analytical picture 
fits very well with the idea of considering type $B$ random variables as
infinitesimal deformations to ordinary non-commutative random variables.
\end{abstract}

\bigskip

Free probability theory was introduced by D. Voiculescu in the eighties (see e.g. \cite{VDN}).
Already in his early work \cite{DVJFA} Voiculescu has found analytical ways for computation of 
a number of natural operations in his theory, such as free convolution.  Later on, Speicher 
\cite{RS2} found that the combinatoris of free probability theory has to do with ({}``type $A$'') non-crossing partitions.  
For example, he found a description of
the relation between moments and cumulants in terms of the lattice
$NC^{(A)}(n)$ of noncrossing partitions of $\{1,2,\dots,n\}$; free independence could then be
phrased in terms of  vanishing of mixed cumulants.
We refer the reader to 
\cite{NSbook} for a detailed description.

This paper is devoted to the exploration of a new notion of independence,
called type $B$ freeness, which was introduced by Biane, Goodman and
Nica in \cite{BGN}. The original motivation for the introduction of
this notion was the fact that the lattice of non-crossing partitions,
central to (ordinary, or {}``type $A$'') probability theory is naturally
associated to the symmetric groups, which are Weyl groups of Lie groups
of type $A$. If the symmetric group is replaced by the hyperoctohedral
group (the Weyl group of a type $B$ Lie group), one obtains the lattice
$NC^{(B)}(n)$ of type $B$ non-crossing
partitions \cite{reiner}, and thus one seeks a probabilty theory whose
underlying combinatorics is governed by this other lattice. In \cite{BGN},
the authors have shown that type $B$ free probability theory does indeed
exist: one can make sense of a type $B$ law and of a type $B$ non-commutative
random variable. There is a notion of freeness, which has the desired
combinatorial structure. Finally, one has the notions of type $B$ free
convolutions of type $B$ laws, together with an appropriate linearizing 
transform (the $R$-transform of type $B$,) and so on. 
Later in \cite{MVP}, M. Popa
showed
that there is a natural notion of type $B$ semicircular law and a type
$B$ central limit theorem.

Type $B$ freeness can be considered unusual in that there seemed
to be no obvious notion of positivity. For a single random variable
of type $B$, its law can be viewed as being described by a pair of measures
$(\mu,\mu')$. Unfortunately, although it is clear that $\mu$ should
be positive, there was no obvious positivity condition on $\mu'$;
and indeed, the measure $\mu'$ associated to a type $B$ semicircular
variable need not be positive (as remarked in \cite{MVP} for the central
limit and Poisson limit distributions). To find a reasonable
positivity assumption, we chose to introduce a notion of \emph{infinitesimal
law} of a family of random variables, which is a weakening of the
notion of a type $B$ law (this notion is almost implicit in the work
of Biane, Goodman and Nica; indeed, they show that type $B$ probability
is related to freeness with amalgamation over the algebra 
$\mathbb{C}+\mathbb{C}\hbar$ of power series in $\hbar$ taken modulo 
terms of order $\hbar^{2}$ or higher). 

More precisely, there is an infinitesimal law associated to every
family of type $B$ random variables (though some informationis lost
when passing to the infinitesimal law). This weakening, however, is
of no consequence in a single-variable case, and amounts to
interpreting
the pair $(\mu,\mu')$ as the zeroth and first derivative of a family
of laws $\mu_{t}$ (i.e., $\mu=\mu_{0}$ and 
$\mu'(f)=\frac{d}{dt}\Big|_{t=0}\mu_{t}(f)$
for sufficiently nice $f$). One natural notion of positivity is then
to require existence of a family $\mu_{t}$ of \emph{positive} laws
whose derivative is $\mu'$. One can then check that the obvious notion
of infinitesimal freeness (which requires that freeness conditions
are fulfilled to order $o(t)$) are compatible with type $B$ freeness.
In particular, it turns out (Theorem \ref{inf}) that free convolution 
of type $B$ is intimately
related to free convolution of type $A$: 
if $(\eta,\eta')=(\mu,\mu')\boxplus_{B}(\nu,\nu')$
then $\eta=\mu\boxplus\nu$ and $\eta'=\frac{d}{dt}(\mu_{t}\boxplus\nu_{t})$
(usual free convolution), where $\mu_{t}$ and $\nu_{t}$ are any
two families of laws having as their derivatives at $t=0$ $\mu'$
and $\nu'$, respectively. 

Although the notions of infinitesimal law, infinitesimal freeness
add to the proliferation of different notions of non-commutative random
variables, freeness and so on, we feel that these notions are justified
since they simplify our presentation and look rather natural. For
example, the rather mysterious type $B$ semicircular law is nothing
by the infinitesimal law associated to the family of laws of variables
$x+ty$ where $x$ and $y$ are two free semcirculars (Example \ref{24}).

The notion of infinitesimal freeness is in spirit related to the notion of second-order freeness 
introduced by Mingo and Speicher \cite{MS2ndOrder}, but is different from it (our notion is related to a first derivative, and Speicher's is related to a second derivative, 
defined in the case the first derivative vanishes).

A very rich source of infinitesimal laws is given by random matrix
theory. Indeed, if $X_{N}$ is an $N\times N$ random matrix, its
moments typically have an expansion in powers of $1/N$. Keeping the
zeroth and first order terms in $1/N$ then gives rise to an infinitesimal
law. Unfortunately, we were unable to find a direct connection between
ordinary random matrices and type $B$ freeness. The natural guess ---
taking $X_{N}$ to be a \emph{real }Gaussian random matrix and looking
at its law to order $1/N$ as $N\to\infty$ does not produce an infinitesimal
(i.e., type $B$) semicirclar variable. Indeed, as was shown by 
\cite{Johanssen, Edelmann},
the law of an $N\times N$ real random matrix is approximately
$$ \mu_t = \frac{2}{\pi \sqrt{1-t^2}} + t \left( \frac{1}{4}(\delta_1 + \delta_{-1}) - \frac{1}{2\pi \sqrt{1-t^2}} \right),\qquad t = N^{-1}.$$
On the other hand, the type $B$ semicircular law is associated with the following infinitesimal law:
$$\mu_t = \mu_0 + t(\mu_0 - \nu),$$ where $\mu_0$ is the semicircle law and $\nu$ is the arcsine law.  Because of the 
type $B$ free central limit theorem, it is clear that such independent real
Gaussian matrices become in any way asymptotically B-free. 

We show, however, that if we instead start with an $N\times N$ self-adjoint
matrix $X_{N}$ whose entries are semicircular variables, then the
infinitesimal law associated to $X_{N}$ by keeping expressions of
order $0$ and $1$ in $1/N$ \emph{does} converge to a type $B$ semicircular
variable (in fact, the law of $X_{N}$ is semicircular of variance
$(1+1/N)$ for all $N$). Note that $X_{N}$ is {}``symmetric''
in the sense that it is equal to the matrix obtained from $X_{N}$
by transposing all rows and columns, and is thus a free probability
analog of a real Gaussian random matrix (Corollaries 
\ref{cor:lawofXNis} and \ref{cor:infLawXN}.) The matrix $X_{N}$ no 
longer posesses a unitary symmetry, but rather an orthogonal one. It 
would be interesting to investigate if a direct connection to the 
combinatorics of the hyperoctahedral group (related to the orthogonal 
Lie groups) can be made in this way.

One of the main goals of this paper is to give an 
analytic framework for type $B$ free probability. Namely, we show that the operation 
$\boxplus_B$ of free additive  convolution of type $B$ is well defined 
on a product of two spaces,
the first of which is the space of Borel probability measures on $
\mathbb R$, and the second is essentially the space of distributions
on the real line which are derivatives of positive finite measures, not
necessarily probability measures.
(We find in fact three such second coordinate spaces which are stable
under $\boxplus_B$.) 
In this context, it turns out that $\boxplus_B$ is a re-casting
in terms of derivatives of Boolean cumulants of another operation,
$\boxplus_C$, the conditionally free convolution introduced by
Bo\.{z}ejko, Leinert and Speicher in \cite{BLS}. We prove this
result in Theorem \ref{ans}, thus answering an open problem from
\cite{BGN}. Surprisingly, there is no counterpart of this 
correspondence for multiplicative convolutions.

The paper is organized in four sections. In Section 1 we introduce the
main notions and tools used in our proofs. Section 2 establishes the 
connection between infinitesimal freeness and freeness of type $B$, for
both single and multi-variable cases. Section 3 is dedicated to
the description of free additive convolution of type $B$ from
an analytic perspective. In Section 4 we use results form Sections 2
and 3 to describe some limit measures, namely several type $B$ stable 
distributions and the type $B$ Poisson distribution. In Section 4 we
provide a matricial model for type $B$ freeness, and in Section 5
we discuss the operation $\boxtimes_B$ of free multiplicative convolution of type $B$.

\bigskip

\noindent{\bf Acknowledgements}. We are indebted to Teodor Banica and Edouard Maurel-Segala 
for numerous useful discussions.
STB would also like to thank Jacek Szmigielski for 
help in the choice of an appropriate space for the second
coordinate of the free additive convolution of type $B$.

\section{Notations and preliminary results}

The key concept for our paper is presented in the following definition \cite[Definition 6.1]{BGN}:
\begin{definition}
A noncommutative probability space of type $B$ is a system 
$(\mathcal A,\tau,\mathcal V,\varphi,\Phi)$, where
\begin{enumerate}
\item $(\mathcal A,\tau)$ is a type $A$ noncommutative probability space
(i.e. $\mathcal A$ is a unital algebra over $\mathbb C$ and $\tau$ is a
linear functional carrying the unit of the algebra into 1);
\item $\mathcal V$ is a complex vector space and $\varphi\colon\mathcal V\to\mathbb C$ is a linear functional;
\item $\Phi\colon\mathcal A\times\mathcal V\times\mathcal A\to\mathcal V$ is a two-sided action of $\mathcal A$ on $\mathcal V$. We will denote $\Phi(a,\xi,b)$ simply by $a\xi b$, for any $a,b\in\mathcal A,\xi\in\mathcal V$.
\end{enumerate}
\end{definition}
For our purposes, we will need additional structure. Thus, from now on we will assume that $\mathcal A$ is a $C^*$-algebra, $\tau$ is positive, $\mathcal V$ is seminormed, $\varphi$ is continuous, and the action
$\Phi$ is separately continuous.

It was observed in \cite{BGN} that one can define a structure of unital algebra on $\mathcal A\times\mathcal V$ as follows. We represent any 
vector $(a,\xi)\in\mathcal A\times\mathcal V$ as
$\left[\begin{array}{cc}
a & \xi\\
0 & a
\end{array}\right].$ Then 
$$(a_1,\xi_1)\cdot(a_2,\xi_2)=\left[\begin{array}{cc}
a_1 & \xi_1\\
0 & a_1
\end{array}\right]\left[\begin{array}{cc}
a_2 & \xi_2\\
0 & a_2
\end{array}\right]=\left[\begin{array}{cc}
a_1a_2 & a_1\xi_2+\xi_1a_2\\
0 & a_1a_2
\end{array}\right]=(a_1a_2,a_1\xi_2+\xi_1a_2).$$
The unit is simply $(1,0)$, where $1$ is the unit of $\mathcal A$.
As observed in \cite{MVP}, we can view $\mathcal A\times\mathcal V$
as an operator-valued noncommutative 
probability space over the scalar (commutative) algebra

\begin{equation}\label{c}
\mathcal C=\left\{\left[\begin{array}{cc}
z & w\\
0 & z
\end{array}\right]\colon z,w\in\mathbb C\right\}
\subseteq\mathcal M_2(\mathbb C)
\end{equation}
with the conditional expectation $E((a,\xi))=(\tau(a),\varphi(\xi)).$
(As observed in the introduction, this algebra is isomorphic to the
algebra $\mathbb C+\hbar\mathbb C$ of power series taken modulo terms
of order higher than two.)
Since $\mathcal C$ is in the centre of $\mathcal A\times\mathcal V$, the notion of joint moments generalizes in the obvious way to the context of probability space over $\mathcal C$ (see also \cite{RS2}
and \cite{MVP}).
We will next define the non-crossing cumulants of type $A$ and $B$.
\begin{definition}\label{typeAcumul}
The type $A$ free cumulants associated to a noncommutative probability
space $(\mathcal A,\tau)$ are the family of multilinear functionals
$\left(\kappa_n^{(A)}\colon\mathcal A^n\to\mathbb C\right)_{n=1}^\infty$ determined by the equation 
\begin{equation}\label{eq1}
\sum_{\pi\in NC^{(A)}(n)}\prod_{F\in\pi}\kappa_{|F|}^{(A)}\left((a_1,\dots,a_n)|F\right)=\tau(a_1\cdots a_n),
\end{equation}
for all $a_1,\dots,a_n\in\mathcal A$, $n\in\mathbb N$. Here, $F\in\pi$
means {}``$F$ is a block of $\pi$'', $|F|$ denotes the cardinality of the block $F$, and if $F=\{j_1<j_2<\cdots<j_{|F|}\}$, the notation 
$\kappa_{|F|}^{(A)}\left((a_1,\dots,a_n)|F\right)$ means
$\kappa_{|F|}^{(A)}(a_{j_1},a_{j_2},\dots,a_{j_{|F|}}).$
\end{definition}
We observe immediately that for a given $k$-tuple of random variables
${\bf a}=(a_1,a_2,
\dots,a_k)\in\mathcal A^k$, the free cumulants of type $A$ of {\bf a} 
represent simply a family of numbers unique determined by the joint moments of ${\bf a}$ and relation \eqref{eq1}.

Following \cite{BGN}, we define the free cumulants of type $B$ the following way:
\begin{definition}\label{typeBcumul}
Let $(\mathcal A,
\tau,\mathcal V,\varphi,\Phi)$ be a noncommutative probability space of 
type $B$. The type $B$ free cumulants associated to it are the family of multilinear $\mathcal C$-valued functionals $\left(\kappa_n^{(B)}\colon
(\mathcal A\times\mathcal V)^n\to\mathcal C\right)_{n=1}^\infty,$ uniquely determined by the
equation
\begin{equation}\label{eq2}
\sum_{\pi\in NC^{(A)}(n)}\prod_{F\in\pi}\kappa_{|F|}^{(B)}\left((a_1,\xi_1),\dots,(a_n,\xi_n)|F\right)=E((a_1,\xi_1)\cdots(a_n,\xi_n)),
\end{equation}
where the notations correspond to the ones in 
Definition \ref{typeAcumul}.
\end{definition} 
Thus, the free cumulants of type $B$ are defined according to the same equation as the ones of type $A$, but viewed over the algebra $\mathcal C
.$ If one looks at the second coordinate of $\kappa_{n}^{(B)}\left((a_1,\xi_1),\dots,(a_n,\xi_n)
\right)$, one sees that, while the first coordinate equals the type $A$ 
free cumulant, the second is quite different.

The definition of freeness of type $B$ coincides thus with the definition of freeness of type $A$ for the first coordinate (i.e. for the pair $(\mathcal A,\tau)$) and the novelty element is brought by the second coordinate. More precisely,

\begin{definition}\label{freenessB}
Let $(\mathcal A,\tau,\mathcal V,\varphi,\Phi)$ be noncommutative probability space of type $B$. Let $\mathcal A_1,\dots,\mathcal A_k$
be unital subalgebras of $\mathcal A$ and $\mathcal V_1,\dots,\mathcal 
V_k$ be linear subspaces of $\mathcal V$ so that $\mathcal V_j$ is invariant under the action of $\mathcal A_j$, $1\leq j\leq k$. We say
that $(\mathcal A_1,\mathcal V_1),\dots,(\mathcal A_k,\mathcal V_k)$ 
are freely independent in $(\mathcal A,\tau,\mathcal V,\varphi,\Phi)$
if the following happens:
\begin{enumerate}
\item[{\rm (i)}] For any $n\in\mathbb N$, $i_1\neq i_2\neq\cdots\neq i_n\in\{1,\dots,k\}$ if $a_j\in\mathcal A_{i_j}$ satisfy $\tau(a_{j})=0$ for all $1\leq j\leq n$, then $\tau(a_1\cdots a_n)=0$ (i.e. $\mathcal
A_1,\dots,\mathcal A_k$ are free in the type $A$ sense in $(\mathcal A,
\tau)$.)
\item[{\rm(ii)}] The formula
\begin{equation}
\varphi(a_m\cdots a_1\xi b_1\cdots b_n)=\left\{\begin{array}{lrr}
0 & {\rm if\ }m\neq n & \\
\delta_{i_1,j_1}\cdots\delta_{i_n,j_n}\tau(a_1b_1)\cdots\tau(a_nb_n)\varphi(\xi) & {\rm if\ }m=n & 
\end{array}\right.
\end{equation}
holds whenever
\begin{trivlist}
\item[$\bullet$] $m,n\in\mathbb N$, and $i_m,\dots,i_1,h,j_1\dots,j_n\in\{1,\dots,k\}$ are such that any two consecutive indices in the list are different from each other;
\item[$\bullet$] $a_m\in\mathcal A_{i_m},\dots,a_1\in\mathcal A_{i_1},\xi\in\mathcal V_h,b_1\in\mathcal A_{j_1},\dots,b_n\in\mathcal A_{j_n}$ are such that $\tau(a_1)=\cdots=\tau(a_m)=0=\tau(b_1)=\cdots=\tau(b_n).$
\end{trivlist}
\end{enumerate}

\end{definition}
This definition appears as Definition 7.2 in \cite{BGN}. Obviously, two 
type $B$ random variables are free if they live in pairs (algebra, linear space) which are B-free.

\begin{example}\label{AB}
{\rm (We thank Beno\^{i}t Collins for indicating this example to us.)
A simple, and yet very useful, example can be obtained from type $A$
freeness. Consider the type $B$ probability space $(\mathcal A,
\varphi,\mathcal A,\varphi,\Phi)$, where $\Phi$ is the action by
multiplication from the left and by multiplication with elements
from $\mathcal A^{\rm op}$ from the right, and $(\mathcal A,\varphi)$
is simply a type $A$ noncommutative probability space (we will denote
such a space just by $(\mathcal A,
\varphi,\mathcal A,\varphi)$.) It follows
easily from the definition of type $A$ freeness that if $\mathcal A_1,
\dots,\mathcal A_k$ are $A$-free in $(\mathcal A,\varphi)$, then
$(\mathcal A_1,\mathcal A_1),\dots,(\mathcal A_k,\mathcal A_k)$ are 
$B$-free in $(\mathcal A,
\varphi,\mathcal A,\varphi).$
In particular, if $x,y \in\mathcal A$, then $(x,y)\in\mathcal A\times\mathcal A
$ is a type $B$ random variable.  }
\end{example}

For the purpose of studying free convolutions of type $B$, the following two observations are essential. First, in the world of {\em formal} power series with coefficients in $\mathcal C$, the type $B$ $R$-transform of a type $B$ random variable linearizes type $B$ free additive 
convolution: if $(a_1,\xi_1),(a_2,\xi_2)$ are free, then 
$R_{(a_1,\xi_1)}(z)+R_{(a_2,\xi_2)}(z)=R_{(a_1+a_2,\xi_1+\xi_2)}(z)$,
where 
$$R_{(a_j,\xi_j)}(z)=\sum_{n=1}^\infty\kappa_n^{(B)}(\underbrace{(a_j\xi_j),\dots,(a_j,\xi_j)}_{n \ \rm times})z^n.$$
Second, the type $B$ combinatorial structure behaves identically to the 
type $A$ structure viewed over the algebra $\mathcal C$. 
These results appear in \cite{BGN}, as Theorem 7.3
and Proposition 6.5.

An immediate consequence of these results is that all formal power
series operations and properties used in type $A$ free probability have 
a natural counterpart with the same properties in type $B$ free
probability when replacing complex numbers with elements from $\mathcal 
C$ as coefficients (more details below). This fact has been used 
massively in \cite{BGN}
in order to describe, among others, type $B$ distributions of sums (see 
above) and products of of $B$-free random variables. Of course, these
correspond to certain operations, defined on the space of type $B$ 
distributions, which we will call free additive (respectively 
multiplicative) convolution of type $B$, and denote them $\boxplus_B$
and $\boxtimes_B$, respectively. In order to better 
capture the analytic structure of type $B$ convolutions, it is 
helpful  
to view 
the operations on formal power series corresponding to $\boxplus_B$
and $\boxtimes_B$ as operations on analytic maps from $\mathcal C$ to
itself. 
The following setup turns out to be the appropriate 
one:
Define 
$$Z=\left[\begin{array}{cc}
z & w\\
0 & z
\end{array}\right]\in\mathcal C.$$
We immediately observe that 
$$\left[\begin{array}{cc}
z & w\\
0 & z
\end{array}\right]\left[\begin{array}{cc}
z' & w'\\
0 & z'
\end{array}\right]=\left[\begin{array}{cc}
zz' & zw'+wz'\\
0 & zz'
\end{array}\right] ,\quad z,z'w,w'\in\mathbb C,$$
and thus
$$Z^n=\left[\begin{array}{cc}
z^n & nz^{n-1}w\\
0 & z^n
\end{array}\right]\quad\forall n\in\mathbb N,
\quad Z^{-1}=\left[\begin{array}{cc}
z^{-1} & -wz^{-2}\\
0 & z^{-1}
\end{array}\right].$$

For any formal power series in one variable $\tilde f(z)=\sum_{n=0}^\infty A_n z^n$, with
coefficients $A_n=\left[\begin{array}{cc}
a_n & b_n\\
0 & a_n
\end{array}\right]\in \mathcal C$, as defined in
\cite[Section 5.2]{BGN}, we 
define the function 

$$f(Z)=\sum_{n=0}^\infty A_n Z^n,\quad f\colon \Omega\subseteq\mathcal C\to\mathcal C,$$
where $\Omega$ is some open set in $\mathcal C$.
Such an $f$ typically needs not make sense for 
any $Z\in\mathcal C$, so we must analyze carefully the two components 
of this complex map. We have
$f=(f_1,f_2)=\left[\begin{array}{cc}
f_1 & f_2\\
0 & f_1
\end{array}\right]$, where $$f_1(z,w)=f_1(z)=\sum_{n=0}^\infty a_n z^n, \quad
{\rm and }\quad
f_2(z,w)=\sum_{n=0}^\infty na_n z^{n-1}w+\sum_{n=0}^\infty b_n z^n=
wf_1'(z)+g(z).$$
We observe immediately that when we compare this to the formal power
series, we obtain 
$$\tilde f(z)=\left(\sum_{n=0}^\infty a_nz^n,\sum_{n=0}^\infty b_nz^n\right)=(\tilde f_1(z),\tilde g(z)).$$
It is to be noted that only the first coordinate of $Z$ is restricted
by possible issues related to the radius of convergence of $f_1,g$. The
natural power series operations obey this correspondence. The main
advantage of using maps of $\mathcal C$ instead of formal power series
is that it allows a notion of inverse map with respect to 
composition.
Indeed, although the variable $w$ seems to play a minor, rather inconvenient, role, in many circumstances it will allow us to define
the {\em inverse with respect to composition} of a function $f$, noted $f^{-1}$,
by the obvious condition $f\circ f^{-1}(Z)=Z$ and $f^{-1}\circ f(Z)=Z$. We immediately 
observe that for $f$ to be invertible we need $f_1$ as a one-variable function to be 
invertible and then, if $z=f^{-1}_1(\zeta)$ and $\Xi=(\zeta,v),$ we get 
$$f(Z)=\Xi\iff f^{-1}(\Xi)=Z\iff z=f_1^{-1}(\zeta) {\rm\ \  and\ \ } w=\frac{v-g(f_1^{-1}(\zeta)
)}{f_1'(f_1^{-1}(\zeta))}.$$ Thus, whenever the one-variable complex function $f_1$ 
is well defined and (locally) invertible, and $g$ is well defined, we can define the 
composition inverse $f^{-1}$ of $f$.

Now let us observe that the formal power series $\tilde f$
is uniquely determined by the function $f$ whenever $f$ 
has a nonzero radius of convergence (meaning $f_1 $ and $g$ have, as one-variable
complex analytic functions, non-zero radius of convergence each). Indeed, knowing
$f$ on any open set in $\mathcal C$ means knowing in particular $f_1$ and $g$
in some open set in $\mathbb C$, and hence knowing the coefficients $a_n,b_n$
for all $n\in\mathbb N$, i.e. knowing all $A_n$. Moreover, the set of such functions
forms an algebra which is 
commutative and embeds in the algebra of formal power series with coefficients in $\mathcal C$. (Observe that the embedding is given by
$(f_1,f_2)\mapsto(f_1(z),f_2(z,0))$.)
The unit is the constant function $e(z,w)=(1,0)$. 

We now use Theorem 16.15 in \cite{NSbook}:

\begin{theorem}
For $f,g$ formal power series in the $s$ non-commuting variables $z_1,
\dots,z_s$, without free term, with coefficients in a commutative algebra, the following two conditions are equivalent:
\begin{equation}
{\rm Cf}_{(i_1,\dots,i_n)}(g)=\sum_{\pi\in NC^{(A)}(n)}{\rm Cf}_{(i_1\dots,i_n);\pi}(f)\quad\forall
n\ge1,1\le i_1,\dots,i_n\le s.
\end{equation}
\begin{equation}
g=f(z_1(1+g),\dots,z_s(1+g)).
\end{equation}
\end{theorem}
(Here ${\rm Cf}_{(i_1\dots,i_n);\pi}(f)$ denotes the product of coefficients of $f$
indexed by the blocks of the partition $\pi$, where the coefficient corresponding
to the block $\beta=\{j(1),\dots,j(r)\}$ is ${\rm Cf}_{(i_{j(1)}\dots,i_{j(r)})}(f)$, the 
coefficient of $z_{i_{j(1)}}\cdots z_{i_{j(r)}}$ in the expression of $f$.) We apply this theorem in the case $s=1$ to formal power series with coefficients in $\mathcal C$ to obtain

\begin{corollary}\label{functionalequation}
The complex functions of $Z=\left[\begin{array}{cc}
z & w\\
0 & z
\end{array}\right]$ defined by $M_{(a,\xi)}(Z)=\sum_{n=1}^\infty
E((a,\xi)^n)Z^n$ and $R_{(a,\xi)}(Z)=\sum_{n=1}^\infty \kappa_n^{(B)}(\underbrace{(a,\xi),\dots,(a,\xi)}_n)Z^n$ satisfy the usual moment-cumulant functional equation 
$$M_{(a,\xi)}(Z)=R_{(a,\xi)}(Z({\bf 1}+M_{(a,\xi)}(Z))),$$
where ${\bf 1}=(1,0)=\left[\begin{array}{cc}
1 & 0\\
0 & 1
\end{array}\right]$. Moreover, the functions $M_{(a,\xi)}$ and $R_{(a,\xi)}$ determine
each other uniquely via the above functional equation and the condition that their
degree zero term is zero.
\end{corollary}

\begin{remark}\label{cauchy}
{\rm It will be more convenient for us to work with the appropriate generalization of the Cauchy transform:
$$G_{(a,\xi)}(Z)=\sum_{n=0}^\infty A_nZ^{-n-1}=\left(\sum_{n=0}^\infty \frac{a_n}{z^{n+1}},w\sum_{n=0}^\infty\frac{-(n+1)a_n}{z^{n+2}}+\sum_{n=0}^\infty\frac{b_n}{z^{n+1}}\right),$$
where $A_0=(1,0)$, $a_j=\tau(a^j),$ $b_j=\sum_{k=1}^j\varphi(a^{k-1}\xi
a^{j-k}).$ We will denote the first component by $G_a(z)$ and the function $\sum_{n=0}^\infty\frac{b_n}{z^{n+1}}$ by $g_\xi(z)$ (or, when they will be viewed as corresponding to a distribution pair $(\mu,\nu)
$ having the corresponding moments $(a_n,b_n)$, by $G_\mu(z)$ and 
$g_\nu(z)$).}
\end{remark}

It will be seen later that, while for the first coordinate of a
type $B$ distribution the appropriate object is indeed a Borel
probability measure on the real line, there are several appropriate
possible choices for the second. The definition below will provide
the largest convenient framework, from which we will particularize
as needed.

\begin{definition}\label{mare}
Denote $\mathcal M$ the set of all Borel probability measures on 
$\mathbb R$ and $\mathcal M_0$ the set of linear functionals $\nu$ 
defined on the complex vector space generated by
the functions $\mathbb R\ni t\mapsto(z-t)^{-n}$, $n\ge0$, $\Im z\neq0$,
satisfying the following conditions:
(i) $\nu(1)=0,$ (ii) the correspondence $z\mapsto\nu((z-t)^{-n})$ is 
analytic and satisfies (a) $\nu((\bar{z}-t)^{-n})=\overline{\nu((z-t)^{-n})
},$ and (b) $\lim_{y\to+\infty}(iy)^n\nu((iy-t)^{-n})=0$. We will define the support of $\nu$ to be the complement of the
(open) subset of $\mathbb C\cup\{\infty\}$ on which 
$z\mapsto\nu((z-t)^{-n})$ has a
unique analytical extension satisfying (a). 
\end{definition}
(As an example, Schwarz distributions 
with compact support on $\mathbb R$ satisfying condition (i),
satisfy also (ii)
.) Also, denote $\mathbb C^+
=\{z\in\mathbb C\colon\Im z>0\}$, the upper half-plane of the complex plane.

There are several justifications for defining $\mathcal M_0$ this way,
in terms of the functions $t\mapsto(t-z)^{-n}$ (beyond 
the obvious reason that {}``it works''.) As the reader might 
recall from the introduction, one expects the second coordinate of
a type $B$ distribution to be in a certain sense a derivative of a 
probability measure. The theory of distributions allows one to extend 
the notion of derivative to more general objects than differentiable
functions, roughly by using integration by parts. As a relevant 
example, for a compactly supported finite measure $\nu$ on $\mathbb R$, 
if $f$ is a smooth enough function, then $\int f'\,d\nu=-\int f\,d\nu'.
$ This kind of derivative is known as {\em distributional derivative. }
(The reader can find more information about the origins of this very
rich subject in \cite{schwartz}.) On the other hand, possibly up to a sign, the derivative of 
$(t-z)^{-n}$ with respect to $z$ coincides with the derivative with 
respect to $t$. Thus, {}``philosophically'' one can view in the above 
definition the requirement that $\nu$ is defined on all functions 
$t\mapsto(t-z)^{-n}$ together with (b), (ii) as knowing (and allowing) 
the derivatives of all orders of $\nu$, condition (a), (ii) as 
requiring that $\nu$ is in a certain sense real, and condition
(i) as conveniently generalizing the demand that $\nu$ is a derivative
of a probability on $\mathbb R$. Requiring that $\nu$ is defined on
functions $t\mapsto(t-z)^{-n}$ for all $z\in\mathbb C\setminus\mathbb R
$ corresponds to requiring that $\nu$ be supported on the real line. 
When $\nu$ is the distributional derivative of a Borel probability
measure on $\mathbb R$, the above statements are precisely true.

Also, finite measures $\nu$ on the real line are fully recoverable
from the Cauchy transform $z\mapsto\int(t-z)^{-1}\,d\nu(t),$ $\Im z
\neq0$, in terms of the nontangential limit of this function at points 
of the real line. This fact has been
used with considerable success before in free probability, for ex.
in \cite{DVJFA,BVIUMJ,fe1,MZ} etc. (For a very rich and detailed 
discussion of this problem, we refer the reader to the classical text 
of Akhieser \cite{akhieser}.) However, it is not only finite measures 
on $\mathbb R$ that are described by their actions on functions
$t\mapsto(t-z)^{-1}$. While we do not plan to pursue the subject
here, we will mention that linear functionals defined on spaces
of functions with $L^p$ derivatives can also be recovered from
the nontangential limits at points of $\mathbb R$ of their
evaluation on $t\mapsto(t-z)^{-1}$, as shown in \cite{Luszczki}.

\begin{remark}\label{Ccompact}
{\rm
Let us observe that, when we consider distributions which are compactly 
supported in $\mathbb R$, these objects will be defined, and completely 
described by their action, on monomials $t^n,$ $n\in\mathbb N.$ 
Indeed, this fact is well-known for measures (see, for example, 
\cite[Theorem 2.6.4]{akhieser}). For objects $\nu\in\mathcal M_0$, we 
will show first that one can define the values $\nu(t^n)$, and then 
that these values determine uniquely the functions $z\mapsto
\nu((z-t)^{-n})$. Indeed, since $(t-z)^{-1}=
\sum_{n=0}^\infty t^nz^{-n-1}$ for
$|z|>|t|$, to find $\nu(t)$ one can, using (i) and the standard
methods provided in \cite[Chapter 3]{akhieser}, start by formally 
writing $$\nu(t)=\lim_{z\to\infty}z^2\nu((z-t)^{-1}),\quad
\nu(t^n)=\lim_{z\to\infty}z^{n+1}\left[\nu((z-t)^{-1})-
\sum_{j=0}^{n-1}\nu(t^j)z^{-j-1}\right].$$ Now recalling
the hypothesis of compact support for $\nu$, it follows that 
$z\mapsto\nu((z-t)^{-1})$ is analytic on a neighbourhood of infinity.
From this hypothesis, together with the linearity of $\nu$, it follows 
that all the numbers $\nu(t^n),$ $n\ge0$, are well-defined, and
moreover, they uniquely specify $\nu((z-t)^{-n})$ for all $n\ge0$.
Thus, provided it converges 
on some neighbourhood of infinity, the series $g_\nu(z)= \sum_{n=0}^\infty
\nu(t^n)z^{-n-1}$ uniquely determines $\nu$. 
}
\end{remark}

\bigskip
Next we shall recall some useful facts related to Cauchy transforms of free 
additive convolutions of probability measures on the real line. These 
results have been proved first by Voiculescu in \cite{fe1} under some 
mild restrictions, and later in full generality by Biane in 
\cite{mathz}.
\begin{theorem}\label{subord}
For any measure $\lambda$, let $G_\lambda$ be its Cauchy transform and 
$F_\lambda=\frac{1}{G_\lambda}.$ Assume $\mu_1,\mu_2\in\mathcal M$ and denote $\mu_3=\mu_1\boxplus\mu_2$.
Then there exist two analytic maps $\omega_1,\omega_2\colon\mathbb C^+
\to\mathbb C^+$, uniquely determined by the following properties:
\begin{enumerate}
\item[{\rm(1)}] $G_{\mu_1}(\omega_1(z))=G_{\mu_2}(\omega_2(z))=G_{\mu_3
}(z),$ $z\in\mathbb C^+$;
\item[{\rm(2)}] $\omega_1(z)+\omega_2(z)=z+F_{\mu_3}(z),$ $z\in
\mathbb C^+$;
\item[{\rm(3)}] If $\mu_1,\mu_2$ have compact support, then $\omega_j$ are analytic on some neighbourhood of infinity, $j\in\{1,2\}$;
\item[{\rm(4)}] $\lim_{y\to+\infty}\omega_j(iy)/iy=\lim_{y\to+\infty}
\omega_j'(iy)=1$, $j\in\{1,2\}$.
\end{enumerate}
The above relations determine uniquely the free additive convolution of 
$\mu_1$ and $\mu_2$.
\end{theorem}
Thus, $G_{\mu_3}$ is subordinated to $G_{\mu_1}$ and $G_{\mu_2}$ in
the sense of Littlewood (see \cite[Section 1.5]{Duren}).

\section{type $B$ free probability vs infinitesimal free probability 
theory.}

\subsection{Infinitesimal laws of non-commutative random variables.}

Let $X_{t}^{(j)}$ be a family of non-commutative random variables
in a non-commutative probability space $(\mathcal A,\tau)$. Here $j=1,\ldots,n$
and $t$ is a parameter lying in the set $K\subset\mathbb{R}$ having
zero as an accumulation point. We now define an object that captures
the behavior of the moments of the family $\{X_{t}^{(j)}:t\in K\}$
as $t\to0$ to order $t$ (ignoring orders higher than $t$). 

\begin{definition}
An \emph{infinitesimal law} of $n$ variables is a pair of linear
functionals $\mu,\mu':\mathbb{C}\langle t_{1},\ldots,t_{n}\rangle\to\mathbb{C}$
defined on the algebra of non-commutative polynomials in $n$ indeterminates.
We require that $\mu(1)=1$ and $\mu'(1)=0$.
\end{definition}
The intuitive idea is that if we have a family $X_{t}^{(j)}$ as above,
we would set, for $p\in\mathbb{C}\langle t_{1},\dots,t_{n}\rangle$\begin{eqnarray*}
\mu(p(t_{1},\ldots,t_{n})) & = & \lim_{t\to0}\tau(p(X_{t}^{(1)},\ldots,X_{t}^{(n)})),\\
\mu'(p(t_{1},\ldots,t_{n})) & = & \lim_{t\to0}\frac{1}{t}\left[\tau(p(X_{t}^{(1)},\ldots,X_{t}^{(n)}))-\mu(p(t_{1},\ldots,t_{n}))\right].\end{eqnarray*}

These could also be written as\begin{eqnarray*}
\mu & = & \lim_{t\to0}\mu_{t}\\
\mu' & = & \lim_{t\to0}\frac{1}{t}(\mu_{t}-\mu),\end{eqnarray*}
where $\mu_{t}$ denotes the law of the $n$-tuple $(X_{t}^{(j)}:j=1,\dots,n)$.

Given $\mu,\mu'$, one can define a natural one-parameter family of
laws having the infinitesimal law $(\mu,\mu')$, namely, $\mu_{t}=\mu+t\mu'$
(note that this is a family of laws since $\mu'(1)=0$ and thus $\mu_{t}(1)=1$
for all $t$).

\subsection{Freeness for infinitesimal laws.}

There is also an obvious notion of freeness. Namely, we say that two
families $X_{t}^{(j,1)}$, $j=1,\dots,n_{1}$ and $X_{t}^{(j,2)}$,
$j=1,\dots,n_{2}$ are free (to order $t$) if one has, for arbitrary
polynomials $P_{1},\dots,P_{k}$ \[
\tau\left(\left[P_{1}(\vec{X}_{t}^{(i_{1})})-\tau(P_{1}(\vec{X}_{t}^{(i_{1})}))\right]\cdots\left[P_{k}(\vec{X}_{t}^{(i_{k})})-\tau(P_{k}(\vec{X}_{t}^{(i_{k})}))\right]\right)=o(t)\qquad\textrm{as }t\to0\]
whenever $i_{1},\dots,i_{k}\in\{1,2\}$, $i_{1}\neq i_{2}$, $i_{2}\neq i_{3}$,
$\dots$, and we have set $\vec{X}_{t}^{(i)}=(X_{t}^{(1,i)},\dots,X_{t}^{(n_{i},i)})$.

Thus we make the following definition:

\begin{definition}
Let $(\mu,\mu')$ be an infinitesimal law defined on $n+m$ variables
$t_{1},\dots,t_{n+m}$. We say that $(t_{1},\dots,t_{n})$ and $(t_{n+1},\dots,t_{n+m})$
are \emph{infinitesimally free with respect to $(\mu,\mu')$ }if $(t_{1},\dots,t_{n})$
and $(t_{n+1},\dots,t_{n+m})$ are freely independent to order $t$
with respect to the law $\mu+t\mu'$. 
\end{definition}
It is not hard to see that this condition translates into two requirements:
(i) that $(t_{1},\dots,t_{n})$ and $(t_{n+1},\dots,t_{n+m})$ be
freely independent with respect to the law $\mu$ and (ii) that \begin{equation}
\mu'\left(\left[p_{1}-\mu(p_{1})\right]\cdots\left[p_{k}-\mu(p_{k})\right]\right)-\sum_{j}\mu\left(\left[p_{1}-\mu(p_{1})\right]\cdots\left[\mu'(p_{j})\right]\cdots\left[p_{k}-\mu(p_{k})\right]\right)=0\label{eq:1storder-free}\end{equation}
for any polynomials $p_{1},\dots,p_{k}$ so that $p_{j}\in A_{i(j)}$,
$i(1)\neq i(2)$, $i(2)\neq i(3)$, $\dots$ , where $A_{1}=\mathbb{C}\langle t_{1},\dots,t_{n}\rangle$
and $A_{2}=\mathbb{C}\langle t_{n+1},\dots,t_{m}\rangle$. 

\begin{prop}
\label{pro:restriction-determines}Assume that $A_{1}=\mathbb{C}\langle t_{1},\ldots,t_{n}\rangle$
and $A_{2}=\mathbb{C}\langle t_{n+1},\dots,t_{m}\rangle$ are infinitesimally
free in $\mathbb{C}\langle t_{1},\dots,t_{n+m}\rangle$ with respect
to the law $(\mu,\mu')$. Then the restriction of $(\mu,\mu')$ to
the subalgberas $A_{1}$ and $A_{2}$ determines $(\mu,\mu')$. 
\end{prop}
\begin{proof}
Indeed, this is the case for $\mu$ (because of freeness); and \eqref{eq:1storder-free}
together with linearity and the requirement that $\mu'(1)=0$ defines
$\mu'$ in terms of its restriction to $A_{1}$ and $A_{2}$. 
\end{proof}
\begin{remark}
\label{rem:freeness-fort-implies-inffree}It is immediate from the
definition that if two families $X_{t}^{(j,1)}$, $j=1,\dots,n_{1}$
and $X_{t}^{(j,2)}$, $j=1,\dots,n_{2}$ are actually freely indepedent
(for all $t$), then they are also infinitesimally free. 
\end{remark}
In particular, if we are given two infinitesimal laws $(\mu,\mu')$,
$(\nu,\nu')$, and we set $\mu_{t}=\mu+t\mu'$, $\nu_{t}=\nu+t\nu'$,
then\begin{eqnarray*}
\eta & = & \mu*\nu\\
\eta' & = & \frac{d}{dt}\Big|_{t=0}(\mu_{t}*\nu_{t})\end{eqnarray*}
(here $\mu*\nu$ denotes the free product of two laws) gives rise
to an infinitesimal law $(\eta,\eta')$ of $n+m$ variables so that
its restrictions to the first $n$ and the last $m$ variables are
exactly $(\mu,\mu')$ and $(\nu,\nu')$. We denote this law by $(\eta,\eta')=(\mu,\mu')*(\nu,\nu')$.
This is the (unique because of Proposition \eqref{pro:restriction-determines})
free product of $(\mu,\mu')$ and $(\nu,\nu')$.

\subsection{Connection with type $B$ free independence.}

Let now $(X^{(j)},\xi^{(j)})$ be type $B$ random variables in a non-commutative
type $B$ probability space $(A,\tau,\mathcal{V},f,\Phi)$. Associated
to them we consider an infinitesimal famly of (type $A$) random variables\[
X_{\hbar}^{(j)}=\left[\begin{array}{cc}
X^{(j)} & \xi^{(j)}\\
0 & X^{(j)}
\end{array}\right]=X^{(j)}+\hbar\xi^{(j)},\]
where $\hbar$ is a formal variable satisfying $\hbar^{2}=0$. In
other words, we consider the infinitesimal law $(\mu_{0},\mu_{0}')$
given by

\begin{eqnarray*}
\mu_{0}(t_{i_{1}},\ldots,t_{i_{n}}) & = & \tau(X^{(i_{1})}\cdots X^{i_{n}})\\
\mu'_{0}(t_{i_{1}},\ldots,t_{i_{n}}) & = & \sum_{j}f(X^{(i_{1})}\cdots X^{(i_{j-1})}\xi^{(i_{j})}X^{(i_{j+1})}\cdots X^{(i_{n})}).\end{eqnarray*}
 We shall call this \emph{the infinitesimal law }associated to the
type $B$ family $(X^{(j)},\xi^{(j)}).$ 

Note that $(\mu_{0},\mu_{0}')$ does not capture all of the type $B$
law of the original family but only certain averages of moments (it
does, however, capture the {}``type $A$ part'' of the law of $(X^{(j)},\xi^{(j)})$,
which is exactly $\mu_{0}$). For example,\[
\mu_{0}'(X_{1}X_{2}X_{1}X_{2})=f(X_{1}\xi_{2}X_{1}X_{2})+f(\xi_{1}X_{2}X_{1}X_{2})+f(X_{1}X_{2}\xi_{1}X_{2})+f(X_{1}X_{2}X_{1}\xi_{2}).\]
Even if we assume some tracialy of $f$, e.g. $f(A\xi_{j}B)=f(BA\xi_{j})=f(\xi_{j}BA)$,
the four terms on the right reduce to two terms $2(f(X_{1}X_{2}X_{1}\xi_{2})+f(X_{2}X_{1}X_{2}\xi_{1}))$
but the equation still cannot be used to determine fully the type
B law of the family $(X_{j},\xi_{j})$. 

Nonetheless we have:

\begin{prop}
\label{pro:B-free-implies-inffree}If $((X^{(1)},\xi^{(1)}),\dots,(X^{(n)},\xi^{(n)})$
and $((X^{(n+1)},\xi^{(n+1)}),\dots,(X^{(n+m)},\xi^{(n+m)}))$ are
two families of type $B$ variables in $(A,\mathcal{V},\tau,f,\Phi)$
which are free, then the infinitesimal law $(\eta,\eta')$ associated
to $((X^{(1)},\xi^{(1)}),\dots,(X^{(n+m)},\xi^{(n+m)}))$ is the free
product of the infinitesimal laws $(\mu,\mu')$ and $(\nu,\nu')$
associated to the families $((X^{(1)},\xi^{(1)}),\dots,(X^{(n)},\xi^{(n)}))$
and $((X^{(n+1)},\xi^{(n+1)}),\dots,(X^{(n+m)},\xi^{(n+m)}))$.
\end{prop}
\begin{proof}
We first note that $\eta=\mu*\nu$, because type $B$ freeness entails
(type $A$) freeness of the families $(X^{(1)},\dots,X^{(n)})$ and $(X^{(n+1)},\dots,X^{(n+m)})$. 

Next, define a derivation $D:A\to\mathcal{V}$ given on monomials
by\[
D(X^{(i_{1})}\cdots X^{(i_{k})})=\sum_{j}X^{(i_{1})}\cdots X^{(i_{j-1})}\xi^{(i_{j})}X^{(i_{j+1})}\cdots X^{(i_{n})}.\]
Thus by definition $\mu'=f\circ D$.

Now, for any polynomials $p_{1},\dots,p_{k}$ so that $p_{j}\in A_{i(j)}$,
$i(1)\neq i(2)$, $i(2)\neq i(3)$,$\dots$ , where $A_{1}=\mathbb{C}\langle t_{1},\dots,t_{n}\rangle$
and $A_{2}=\mathbb{C}\langle t_{n+1},\dots,t_{m}\rangle$, we compute\begin{multline*}
\mu'\left(\left[p_{1}-\mu(p_{1})\right]\cdots\left[p_{k}-\mu(p_{k})\right]\right)-\sum_{j}\mu\left(\left[p_{1}-\mu(p_{1})\right]\cdots\left[\mu'(p_{j})\right]\cdots\left[p_{k}-\mu(p_{k})\right]\right)\\
=\sum_{j}f\left(\left[p_{1}-\tau(p_{1})\right]\cdots Dp_{j}\cdots\left[p_{k}-\tau(p_{k})\right]\right)\\
-\sum_{j}\tau\left(\left[p_{1}-\tau(p_{1})\right]\cdots\left[f(Dp_{j})\right]\cdots\left[p_{k}-\tau(p_{k})\right]\right)=0\end{multline*}
because the sums cancel term-by-term owing to type $B$ freeness.
\end{proof}

\subsubsection{Single variable case.\label{sub:Single-variable-case.}}

In the case that we have a single type $B$ random variable $X$ in a
type $B$ non-commutative probability space $(A,\mathcal{V},\tau,f,\Phi)$
satisfying a traciality condition, the infinitesimal law associated
to $X$ determines its type $B$ distribution. 

In the single variable case, $\tau$ is a trace. We will make the assumption
that the linear map $f$ satisfies\[
f(X^{k}\xi X^{l})=f(X^{k+l}\xi)\]
for all $k,l$. In this case, the infinitesimal family $\mu_{\hbar}=\mu_{0}+\hbar\mu_{0}'$
determines $f$ and $\tau$ completely by the formulas:

\[
\tau(X^{n})=\int t^{n}d\mu_{0}(t),\qquad f(X^{j-i}\xi X^{i})=\frac{1}{j+1}\int t^{j+1}d\mu'_{0}(t).\]

\subsection{Free additive convolution for infinitesimal laws.}

Given two infinitesimal laws $(\mu,\mu')$ and $(\nu,\nu')$ defined
on algebras $\mathbb{C}[t_{1}]$ and $\mathbb{C}[t_{2}]$ we define
their infinitesimal free addtive convolution by\[
(\mu,\mu')\boxplus(\nu,\nu')=((\mu,\mu')*(\nu,\nu'))|_{\textrm{Alg}(t_{1}+t_{2})}.\]
In other words, the additive free convolution is the push-forward
(under the addition map $(t_{1},t_{2})\mapsto t_{1}+t_{2}$) of their
free product.

We note that there are two ways to compute this, in view of Proposition
\ref{pro:B-free-implies-inffree} and Remark \ref{rem:freeness-fort-implies-inffree}:

\begin{prop}
\label{pro:infinitesimalbox-vs-boxplus}Let $(\eta,\eta')=(\mu,\mu')\boxplus(\nu,\nu')$.
Then:\\
(a) $\eta=\mu\boxplus\nu$ (ordinary type $A$ free convolution) and
if we set $\mu_{t}=\mu+t\mu'$, $\nu_{t}=\nu+t\nu'$, then $\eta'=\frac{d}{dt}\Big|_{t=0}\mu_{t}\boxplus\nu_{t}$;\\
(b) Let $\mu^{B}$ and $\nu^{B}$ be the type $B$ laws associated to
$\mu$ and $\nu$ as in \S\ref{sub:Single-variable-case.}. Then
$(\eta,\eta')$ is the infinitesimal law associated to $\mu^{B}\boxplus_{B}\nu^{B}$.
\end{prop}

\section{Analytic computation of free additive convolution of type $B$}

\subsection{Type $B$ free additive convolution}

We observe \cite{BGN} that if $(a_1,\xi_1)$ and $(a_2,\xi_2)$ are $B$-free, then the distribution of their sum depends only on the distributions of the two summands. Thus, it is possible to define a type $B$ free 
additive convolution, as an operation on the space of sequences of 
pairs of complex numbers $(a_n,b_n)$, $n\in\mathbb N$, by using the 
moment-cumulant formula given in Definition \ref{typeBcumul} and the
linearizing property of the type $B$ free cumulants. We will denote 
this operation by $\boxplus_B$. However, as in the 
case of the free convolution of type $A$, one would like to find the 
appropriate analytic object which will be stable under $\boxplus_B$.
There are several relevant answers to this question. We shall first
describe in the proposition below the analytic interpretation for
the operation $\boxplus_B$ as described in \cite{BGN}. (For the 
notation $\mathcal M_0$ used below we refer the reader to Definition
\ref{mare}.)

\begin{prop}\label{prop11}
Consider two type $B$ random variables $(a_1,\xi_1),(a_2,\xi_2)$
which are B-free, and are distributed according to $(\mu_1,\nu_1)$
and $(\mu_2,\nu_2)$, respectively.
Assume that $\mu_1,\mu_2\in\mathcal M$ and $\nu_1,\nu_2\in
\mathcal M_0$ are compactly supported on $\mathbb R$.
Denote by $(\mu_3,\nu_3)$ the distribution of $(a_1+a_2,\xi_1+\xi_2).$ 
Then, with the notations from Theorem \ref{subord} and Remark 
\ref{cauchy}, we have
\begin{trivlist}
\item[(a)] $\mu_3=\mu_1\boxplus\mu_2;$
\item[(b)] $g_{\nu_3}(z)=g_{\nu_1}(\omega_1(z))\omega_1'(z)+g_{\nu_2}
(\omega_2(z))\omega_2'(z),$ $z\in\mathbb C^+.$
\end{trivlist}
Moreover, $\mu_3\in\mathcal M$, $\nu_3\in\mathcal M_0$ have compact 
support in $\mathbb R$.
\end{prop}
\begin{proof}
It will follow from the Corollary \ref{functionalequation} and the 
corresponding type $A$ theory that one can find two subordination
functions $\Omega_1,\Omega_2$ so that 
\begin{equation}\label{O}
\Omega_1(Z)+\Omega_2(Z)
=Z+F_{(a_1+a_2,\xi_1+\xi_2)}(Z)\quad{\rm and}\quad 
G_{(a_j,\xi_j)}(\Omega_j(Z))=G_{(a_1+a_2,\xi_1+\xi_2)}(Z), 
j\in\{1,2\};
\end{equation} (recall that we denote
by $F$ the multiplicative inverse of $G$). 
Indeed, we will prove this fact by directly finding a formula
for $\Omega_j$; the proof will provide simultaneously the part
(b) of our proposition. Let us first observe that part (a) follows
directly from \cite[Section 7.2]{BGN}. Moreover, from (a), Theorem \ref{subord},
and Remark \ref{cauchy}, it follows that, if existing, the first 
coordinate of
$\Omega_j$ depends only on the first coordinate of $Z$ and coincides 
with the subordination function $\omega_j$ provided by Theorem \ref{subord}. So denote $\Omega_j(Z)=\Omega_j(z,w)=(\omega_j(z),o_j(z,w)).$
The subordination relation requires then for the second coordinate that
\begin{equation}\label{oj}
wG_{\mu_3}'(z)+g_{\nu_3}(z)=o_j(z,w)G_{\mu_j}'(\omega_j(z))+
g_{\nu_j}(\omega_j(z)).
\end{equation}
The second coordinate of the first relation in \eqref{O}
(the analogue of Theorem \ref{subord} (b)) gives
$$o_1(z,w)+o_2(z,w)=w-\frac{wG_{\mu_3}'(z)+g_{\nu_3}(z)}{G_{\mu_3}(z)^2}=w(1+F_{\mu_3}'(z))-\frac{g_{\nu_3}(z)}{G_{\mu_3}(z)^2}.$$
We shall isolate from the above two equations $o_1$. Indeed, it
follows easily that 
$o_2(z,w)= [wG_{\mu_3}'(z)+g_{\nu_3}(z)-g_{\nu_2}(\omega_2(z))]\cdot
[G_{\mu_2}'(\omega_2(z))]^{-1}.$ Amplifying the right-hand term
by $\omega_2'(z)$ and then replacing in the
equation above yields

$$o_1(z,w)+\frac{wG_{\mu_3}'(z)\omega_2'(z)+g_{\nu_3}(z)\omega_2'(z)-g_{\nu_2}(\omega_2(z))\omega_2'(z)}{G_{\mu_2}'(\omega_2(z))\omega_2'(z)}=w(1+F_{\mu_3}'(z))-\frac{g_{\nu_3}(z)}{G_{\mu_3}(z)^2},$$
and thus, by the chain rule and Theorem \ref{subord},
\begin{eqnarray}\label{o1}
o_1(z,w) & = & w(1+F_{\mu_3}'(z)-\omega_2'(z))
+g_{\nu_3}(z)\cdot\frac{F_{\mu_3}'(z)-\omega_2'(z)}{G_{\mu_3}'(z)}
+\frac{g_{\nu_2}(\omega_2(z))\omega_2'(z)}{G_{\mu_3}'(z)}\nonumber\\
& = & w\omega_1'(z)+\frac{g_{\nu_3}(z)(\omega_1'(z)-1)+g_{\nu_2}(\omega_2(z))\omega_2'(z)}{G_{\mu_3}'(z)}.
\end{eqnarray}

This provides the complete formula for the subordination function
$\Omega_1$. To conclude the proof of part (b) of the
proposition, one needs only to replace the above formula for $o_1$
in \eqref{oj}, for $j=1$:
$$wG_{\mu_3}'(z)+g_{\nu_3}(z)
=w G_{\mu_1}'(\omega_1(z))\omega_1'(z)
+g_{\nu_1}(\omega_1(z))+G_{\mu_1}'(\omega_1(z))\cdot
\frac{(\omega_1'(z)-1)g_{\nu_3}(z)+g_{\nu_2}(\omega_2(z))\omega_2'(z)
}{G_{\mu_3}'(z)};$$
multiplication by $\omega_1'(z)$ and an application of Theorem 
\ref{subord} yields the desired formula
$$g_{\nu_3}(z)=g_{\nu_1}(\omega_1(z))\omega_1'(z)+
g_{\nu_2}(\omega_2(z))\omega_2'(z).$$
The compacity of the support of $\mu_3$ is known.
It follows from the
compacity of the supports of $\nu_1,\nu_2$, the above formula,
and Theorem \ref{subord} that $g_{\nu_3}$ is analytic on a 
neighbourhood of infinity, that $g_{\nu_3}(\bar{z})=
\overline{g_{\nu_3}(z)}$ and that $\lim_{z\to\infty}zg_{\nu_3}(z)=0$.
Expanding $g_{\nu_3}$ in power series around infinity provides the
values $\nu_3(t^n)$, $n\ge0$, as seen in Remark \ref{Ccompact},
and makes possible to define $\nu_3$ as a linear functional on the
space of functions $t\mapsto(z-t)^{-n}$ in the obvious way.
The property (ii), (b), of Definition \ref{mare} follows easily
from the above and is left as an exercise. This concludes the proof of
our proposition.

\end{proof}

\subsection{A connection with conditionally free convolution.}
We follow next with a rather surprising application of the
result above, namely we show that the type $B$ free additive
convolution encodes, up to translation, the conditionally
free convolution (abbreviated c-free convolution) $\boxplus_C$
introduced by Bo\.zejko, Leinert and Speicher in \cite{BLS}.
This operation is defined on pairs of probability measures on the real 
line $(\mu,\rho)\in\mathcal M\times\mathcal M$, and on the
first coordinate acts as the free additive convolution:
if $(\mu_1,\rho_1),(\mu_2,\rho_2)\in\mathcal M\times\mathcal M$
and we denote $(\mu_3,\rho_3)=(\mu_1,\rho_1)\boxplus_C(\mu_2,\rho_2)$,
then $\mu_3=\mu_1\boxplus\mu_2$. In particular, the Cauchy transforms
for the first coordinate\footnote{It is customary in the theory 
of conditionally free convolution for the {\em second} coordinate to 
be chosen as the one on which $\boxplus_C$ acts as $\boxplus$, the usual free
additive convolution. We have reversed this convention in our paper in
order to emphasize the connection with $\boxplus_B$.} satisfy Theorem \ref{subord}. 

For any $\lambda\in\mathcal M$ we shall denote $h_\lambda(z)=
F_\lambda(z)-z,z\in\mathbb C^+$. It is a consequence of 
\cite[Equation 3.3]{akhieser}
that $h_\lambda$ takes values in the closure of the upper half-plane,
and is real if and only if $\lambda$ is a point mass.
\begin{remark}\label{h}
{\rm The following representation for $h_\lambda$, called the Nevanlinna 
representation, will be used in stating
and proving our next result: for any $\lambda\in\mathcal M$
there exist $a\in\mathbb R$ and a positive finite Borel
measure $\sigma$ on the real line so that

$$h_\lambda(z)=a+\int_\mathbb R\frac{1+tz}{t-z}\,d\sigma(t),\quad
z\in\mathbb C^+.$$
Observing that $\frac{1+tz}{t-z}=(1+t^2)\left(\frac{1}{t-z}-
\frac{t}{1+t^2}\right),$ it follows that one can write
$$h_\lambda(z)=\underbrace{a-\int_\mathbb R\frac{t}{1+t^2}\,(1+t^2)d\sigma(t)}_{\tilde{a}}+\int_\mathbb R\frac{1}{t-z}\,\underbrace{(1+t^2)d\sigma(t)}_{\tilde{\sigma}}=\tilde{a}-G_{\tilde{\sigma}}(z),$$
provided that $\sigma$ has finite second moment 
(i.e. $\int_\mathbb R t^2\,d\sigma(t)<+\infty$.) It is shown in 
\cite[Chapter 3]{akhieser} that this happens whenever $\lambda$ has 
finite second moment
(for the convenience of the reader, we will provide a sketch of the proof in the lemma below.)
Assume for now that this condition holds. Then we can define
an object $\tilde{\sigma}'\in\mathcal M_0$ by the relation
$G_{\tilde{\sigma}'}(z)=-G_{\tilde{\sigma}}'(z)$, i.e.
$$G_{\tilde{\sigma}'}(z)=-G_{\tilde{\sigma}}'(z)=h_\lambda'(z),
\quad z\in\mathbb C^+.$$
In particular, we observe that ${\tilde{\sigma}'}$ is the 
distributional derivative of $\tilde{\sigma}$, and in particular
it indeed belongs to $\mathcal M_0$. For convenience, we shall denote
by $\mathcal M_d$ the space of distributional derivatives of
positive finite measures on the real line.

}
\end{remark}

\begin{lemma}\label{rs}
For any $\sigma\in\mathcal M_d$ there exists a unique $\rho\in
\mathcal M$ so that $\int_\mathbb Rt^2\,d\rho(t)<+\infty$, 
$\int_\mathbb Rt\,d\rho(t)=0$, and $G_\sigma(z)=h_\rho'(z).$
Conversely, for any $\rho\in
\mathcal M$ so that $\int_\mathbb Rt^2\,d\rho(t)<+\infty$, 
$\int_\mathbb Rt\,d\rho(t)=0$, there exists a unique $\sigma
\in\mathcal M_d$ so that $G_\sigma(z)=h_\rho'(z).$
\end{lemma}

\begin{proof}
We mostly follow \cite[Chapter 3]{akhieser}. Let us first observe that for a given positive finite
measure $\rho$, we have $$ \int_\mathbb Rt\,d\rho(t)=
\lim_{y\to+\infty}iy[iyG_\rho(iy)-\rho(\mathbb R)],$$
and
$$ \int_\mathbb Rt^2\,d\rho(t)= \lim_{y\to+\infty}
iy\left[(iy)^2G_\rho(iy)-iy\rho(\mathbb R)-\int_\mathbb Rt\,d\rho(t)\right],$$
provided {\em both} these numbers exist (these results are proved in 
\cite[Theorem 3.2.1]{akhieser}.) In particular, for a probability $\rho$,
we have
\begin{equation}\label{10}
\lim_{y\to+\infty}((iy)^4G_\rho(iy)^2-2(iy)^3G_\rho(iy)+(iy)^2)=
\left(\int_\mathbb Rt\,d\rho(t)\right)^2
.\end{equation} 
Using the monotone convergence theorem, we obtain
\begin{eqnarray}\label{unshpe}
\lim_{y\to+\infty}((iy)^2-2(iy)^3G_\rho(iy)-(iy)^4G_\rho'(iy)) & = &
\lim_{y\to+\infty}\int_\mathbb R(iy)^2-2\frac{(iy)^3}{iy-t}
+\frac{(iy)^4}{(iy-t)^2}\,d\rho(t)\nonumber\\
& = & \lim_{y\to+\infty}\int_\mathbb R\frac{(ity)^2}{(iy-t)^2}\,d\rho(t)\nonumber\\
& = & \int_\mathbb Rt^2\,d\rho(t).
\end{eqnarray}
Combining \eqref{10} and \eqref{unshpe}, we get
\begin{eqnarray}
\lim_{y\to+\infty}(iy)^4(-G_\rho'(iy)-G_\rho(iy)^2) & = & 
\lim_{y\to+\infty}(iy)^2-2(iy)^3G_\rho(iy)-(iy)^4G_\rho'(iy) -[yi(iyG_\rho(iy)-1)]^2\nonumber\\
& = & \int_\mathbb R
t^2\,d\rho(t)-\left(\int_\mathbb Rt\,d\rho(t)\right)^2
\end{eqnarray}
Also, by using Remark \ref{rs}, we obtain that $h_\rho(z)=
a+\int_\mathbb R\frac{1+tz}{t-z}\,d\tau(t).$ Under the hypotheses that
the first moment of $\rho$ is zero and the second moment of $\rho$ is
finite, we claim that $\tau$ must also have finite second moment. 
Indeed, assume this is not the case. According to Section 1 of 
\cite[Chapter 3]{akhieser}, there is a sequence $y_n\to+\infty$
so that $\lim_{n\to\infty}|y_nh_\rho(iy_n)|=\infty.$ But, using the
fact that the first moment of $\rho$ is zero, 
$$\lim_{n\to\infty}|y_nh_\rho(iy_n)|=\lim_{n\to\infty}y_n
\left|\frac{iy_n-(iy_n)^2G_\rho(iy_n)}{iy_nG_\rho(iy_n)}\right|
=\int_\mathbb Rt^2\,d\rho(t)<\infty,$$
providing a contradiction.

Observing that $h_\rho'(z)=\int_\mathbb R\frac{1+t^2}{(t-z)^2}\,d\tau(t)$, we can apply the same methods as before to argue that
\begin{equation}\label{13}
\lim_{y\to+\infty}(iy)^2h'_\rho(iy)=\int_\mathbb R 1+t^2\,d\tau(t).
\end{equation} 
On the other hand,
\begin{eqnarray}\label{14}
\lim_{y\to+\infty}(iy)^2h'_\rho(iy) & = & \lim_{y\to+\infty}(iy)^2(
F'_\rho(iy)-1)\nonumber\\
& = & \lim_{y\to+\infty}\frac{(iy)^4(-G'_\rho(iy)-G_\rho(iy)^2)}{
(iy)^2G_\rho(iy)^2}\nonumber\\
& = & \frac{\int_\mathbb R
t^2\,d\rho(t)-\left(\int_\mathbb Rt\,d\rho(t)\right)^2}{1}
=\int_\mathbb R
t^2\,d\rho(t)<\infty.
\end{eqnarray}
Thus, $d\sigma_0(t):=(1+t^2)d\tau(t)$ is a finite
measure on the real line, and by our previous remark, 
$h'_\rho(z)=-G_{\sigma_0}'(z)=G_{\sigma}(z),$ where $\sigma=
\sigma_0'\in\mathcal M_d.$

Now proving the converse statement is easier. Indeed, if $\sigma\in
\mathcal M_d$, then there exists at least one positive finite
measure $\sigma_0$ so that $\sigma_0'=\sigma$ in distribution. 
Imposing the condition of finiteness specifies this measure
uniquely. So we can define
$\rho_a\in\mathcal M$ by $F_{\rho_a}(z)=a+z-G_{\sigma_0}(z).$
It follows easily from the previous computations that in order
for $\rho_a$ to have first moment zero, it is required that $a=0$.
Similarily, the second moment of $\rho_0$ will be $\lim_{y\to+\infty}
 iyG_{\sigma_0}(iy)=\sigma_0(\mathbb R).$ We will take $\rho=\rho_0$. Details are left to the
reader as an exercise.
\end{proof}
In our proof we will need the following result \cite[Corollary 4]{proc}.

\begin{prop}\label{cfree}
Let $(\mu_1,\rho_1),(\mu_2,\rho_2)\in\mathcal M\times\mathcal M,$
and denote $(\mu_3,\rho_3)=(\mu_1,\rho_1)\boxplus_C(\mu_2,\rho_2)$.
Then 
$$h_{\rho_3}(z)=h_{\rho_1}(\omega_1(z))+h_{\rho_2}(\omega_2(z)),\quad
z\in\mathbb C^+,$$ where $\omega_j$ is the subordination function 
corresponding to $\mu_j$, provided by Theorem \ref{subord} ($j\in
\{1,2\}.$)
\end{prop}

The following theorem answers an open question mentioned in the 
introduction of \cite{BGN}.

\begin{theorem}\label{ans}
Consider pairs $(\mu_j,\rho_j)\in\mathcal M\times\mathcal M$ so that 
$\int_\mathbb Rt^2\,d\rho_j(t)<+\infty$, 
$\int_\mathbb Rt\,d\rho_j(t)=0$, $j\in\{1,2\}.$ Let $\sigma_j\in
\mathcal M_d$
be so that $h_{\rho_j}'(z)=G_{\sigma_j}(z)$, $z\in\mathbb C^+$
(the existence and uniqueness of $\sigma_j$ is guaranteed by Lemma
\ref{rs}.)
Denote  
$(\mu_3,\rho_3)=(\mu_1,\rho_1)\boxplus_C(\mu_2,\rho_2)$
Then $\rho_3$ satisfies the conditions
$\int_\mathbb Rt^2\,d\rho_3(t)<+\infty$, 
$\int_\mathbb Rt\,d\rho_3(t)=0$. Moreover, if we let $(\mu_3,\sigma_3)
=(\mu_1,\sigma_1)\boxplus_B(\mu_2,\sigma_2)$, then $h_{\rho_3}'(z)=
G_{\sigma_3}(z),$ $z\in\mathbb C^+$.
\end{theorem}
\begin{proof} We assume first that
$(\mu_1,\rho_1),(\mu_2,\rho_2)$ have all compact support.
It follows then that $\sigma_1,\sigma_2$ also have
compact support (as described in Definition \ref{mare}.)
Propositions \ref{prop11} and \ref{cfree} guarantee then that 
$$g_{\sigma_3}(z)
=g_{\sigma_1}(\omega_1(z))\omega_1'(z)+g_{\sigma_2}(\omega_2(z))
\omega_2'(z)=(h_{\rho_1}(\omega_1(z))+h_{\rho_2}(\omega_2(z)))'=
h_{\rho_3}'(z).$$
Lemma \ref{rs} concludes the proof for the case of compactly
supported measures. The general case follows by using the denseness
of the set of probabilities with compact support in the space of all 
probabilities and Lemma \ref{rs}.

\end{proof}

\begin{remark}\label{rmk16}
{\rm \begin{enumerate}\item We observe immediately that one can 
generalize 
Theorem 15 to arbitrary pairs of probabilities $(\mu_j,\rho_j)$,
but at the cost of losing a significant analytic object on the second
coordinate in the world of type $B$ distributions. Indeed, relation
(b) in Proposition \ref{prop11}, in which $g_{\nu_j}$ is replaced
by $h_{\rho_j}'$, is easily seen to be stable when we consider weak
limits $\rho_j^{(n)}\to\rho_j$.
\item The theorem above together with the results of Krysztek 
\cite{krysztek} and of
Wang \cite{wang}
on conditionally free infinite divisibility and c-free limit
theorems gives a complete characterization of infinite divisibility
for pairs in $\mathcal M\times\mathcal M_d$.
\item We can also conclude from the above theorem and 
\cite[Proposition 6]{proc} the existence of the type $B$
analogue of the Nica-Speicher partial semigroup with respect to free 
additive convolution. Indeed, if we denote $(\mu_t,\rho_t)=
(\mu,\rho)^{\boxplus_Ct}$, it follows easily from Theorem \ref{ans}
and the corresponding formula $h_{\rho_t}(z)=th_{\rho}(\omega_t(z))$
that $(\mu_t,\sigma_t)=(\mu,\sigma)^{\boxplus_Bt}$ exists for $t\ge1$
and is defined by
\begin{equation}\label{nica-speicher}
\mu_t=\mu^{\boxplus t},\quad g_{\sigma_t}(z)=tg_\sigma(\omega_t(z))
\omega_t'(z),
\end{equation}
for all $(\mu,\sigma)\in\mathcal M\times\mathcal M_d$.
Here the function $\omega_t$ is the subordination function
corresponding to the semigroup $\mu^{\boxplus t}$, provided
by \cite{MZ}: $F_{\mu^{\boxplus t}}(z)=F_\mu(\omega_t(z))$.
We observe immediately that equation \eqref{nica-speicher} allows one 
to 
extend $(\mu,\sigma)^{\boxplus_Bt}$ to pairs $(\mu,\sigma)\in\mathcal M
\times
\mathcal M_0$.
\end{enumerate}}
\end{remark}

Theorem \ref{ans} suggests a more restricted space of
distributions of type $B$ that is stable under $\boxplus_B$.
Consider $f\colon\mathbb R\to\mathbb R$ to be a (bounded)
function of bounded variation, so that $f'(t)dt=df(t)$ is a signed 
finite measure. We shall define
$\nu$ by $g_\nu(z)= \int_\mathbb R\frac{1}{z-t}f{''}(t)\,dt$.
The second derivative of $f$ is taken in the distributional 
sense\footnote{In the most general case, one should not understand the expression 
$\int_\mathbb R\frac{1}{z-t}f{''}(t)\,dt$ as an integral, but as a 
linear functional applied to $t\mapsto(z-t)^{-1}$, as in the definition
of $\mathcal M_0$. The notation $\nu\left[\frac{1}{z-t}\right]$ would 
be more appropriate; however, for our purposes and methods, we believe
our notation to be more suggestive.},
so that $g_\nu(z)=\int_\mathbb R\frac{2}{(z-t)^3}f(t)\,dt.$ We 
observe that one can recover the distribution $\nu$ from the boundary values of
$g_\nu$ in an appropriate topology. For details we refer to
\cite{ultra} (see also \cite{Luszczki}.)
Denote $\mathcal M_2$ the subset of $\mathcal M_0$ formed by
functionals $\nu$ defined this way.

\begin{theorem}\label{thm17}
The set $\mathcal M\times\mathcal M_2$ is stable under the operation
$\boxplus_B$. Moreover, for any $(\mu_j,\nu_j)\in\mathcal M\times
\mathcal M_2$, $j\in\{1,2\}$, with the notation
$(\mu_3,\nu_3)=(\mu_1,\nu_1)\boxplus_B(\mu_2,\nu_2)$ we have $\mu_3=
\mu_1\boxplus\mu_2$ and $g_{\nu_3}(z)=g_{\nu_1}(\omega_1(z))\omega_1'(z)
+g_{\nu_2}(\omega_2(z))\omega_2'(z)$, $z\in\mathbb C^+$, where 
$\omega_j$ are the subordination functions corresponding to $\mu_j$,
$j\in\{1,2\}$.
\end{theorem}
 
\begin{proof} The result follows quite easily from Proposition
\ref{prop11}. Indeed, in order to prove our theorem, we
need to show that, given $f_1,f_2$ generating $\nu_1,\nu_2$
as above, the formula
$g_{\nu_1}(\omega_1(z))\omega_1'(z)+g_{\nu_2}(\omega_2(z))
\omega_2'(z)$ has the form $g_{\nu_3}(z)$ for $\nu_3=f_3{''}$
in the sense of distributions, $f_3$ being a real function with
finite variation. We first argue in the case when $\nu_1,\nu_2$
have compact support.

Consider the primitive $H_{\nu_1}(z)$ of $g_{\nu_1}(z)$ which
satisfies the condition $\lim_{z\to\infty}H_{\nu_1}(z)=0$.
We shall prove that $H_{\nu_1}(z)=\int_\mathbb R\frac{1}{
z-t}\,df_1(t) $, $z\in\mathbb C^+$. Indeed, since 
$H_{\nu_1}'(z)=g_{\nu_1}(z)$, the uniqueness of the primitive
under the hypothesis regarding the behaviour of $H_{\nu_1}$ at
infinity, the definition of distributional derivative and
the relations
$$\left[\int_\mathbb R\frac{1}{
z-t}\,df_1(t)\right]'=-\int_\mathbb R\frac{1}{(
z-t)^2}\,df_1(t)=-\int_\mathbb R\left(\frac{1}{
z-t}\right)'\,df_1(t)=\int_\mathbb R\frac{1}{
z-t}f_1{''}(t)\,dt=g_{\nu_1}(z),$$
complete the proof of our statement.

Let $\sigma$ be a monotone function on the real line with finite 
variance.
Since the function $z\mapsto\int_\mathbb R\frac{1}{\omega_1(z)-t}\,
d\sigma(t)$ maps $\mathbb C^+$ into the lower half-plane and 
$\lim_{y\to+\infty}iy\int_\mathbb R\frac{1}{\omega_1(iy)-t}\,
d\sigma(t)=1$, it follows from Section 1 of \cite[Chapter 3]{akhieser} 
that for any such $\sigma$
and $\omega_1$ as in Theorem \ref{subord}, there exists
another monotone function $\tau$ with finite variance on 
$\mathbb R$ so that $\int_\mathbb R\frac{1}{\omega_1(z)-t}\,
d\sigma(t)=\int_\mathbb R\frac{1}{z-t}\,d\tau(t)$, $z\in\mathbb C^+$.
Moreover, the variance of the two functions coincides (i.e. the 
total mass of the real line is the same under both measures $d\sigma(t)
$ and $d\tau(t)$.)
Since the bounded function $f_1$ takes real values and has finite 
variance, it can be written as the difference of two positive 
nondecreasing functions $f_1^+$ and $f_1^-$; applying the previous
 observation
separately to $f_1^+$ and $f_1^-$ guarantees the existence of a real
valued function $\tilde{f}_1$ with finite variation so that
$H_{\nu_1}(\omega_1(z))=\int_\mathbb R\frac{1}{z-t}\,d\tilde{f}_1(t),$
$z\in\mathbb C^+.$
As $g_{\nu_1}(\omega_1(z))\omega_1'(z)=
H_{\nu_1}'(\omega_1(z))\omega_1'(z)$, we conclude that there exists
$\tilde{\nu}_1=\tilde{f}_1{''}\in\mathcal M_2$ so that 
$g_{\tilde{\nu}_1}(z)=g_{\nu_1}(\omega_1(z))\omega_1'(z)$. 
By the same method we find 
$\tilde{\nu}_2=\tilde{f}_2{''}\in\mathcal M_2$ so that 
$g_{\tilde{\nu}_2}(z)=g_{\nu_2}(\omega_2(z))\omega_2'(z)$. It follows now easily that $f_3=\tilde{f}_1+\tilde{f}_2$ provides
a $\nu_3=f_3{''}\in\mathcal M_2$ so that 
$g_{\nu_3}(z)=g_{\nu_1}(\omega_1(z))\omega_1'(z)+
g_{\nu_2}(\omega_2(z))\omega_2'(z)$, $z\in\mathbb C^+$, and
so, by Proposition \ref{prop11}, $(\mu_3,\nu_3)=
(\mu_1,\nu_2)\boxplus_B(\mu_2,\nu_2).$ This holds for any
$\nu_1,\nu_2\in\mathcal M_2$ with compact support.

The general case follows by approximating $df_j(t)$ with compactly
supported measures $df_j^{(n)}(t)$, $j\in\{1,2\}$.
\end{proof}

For convenience, we shall collect in the corollary below the three
spaces of distributions that have been shown to be stable under
free additive convolution of type $B$. Of course, we do not claim these 
are all possible choices of such spaces, but only the ones that we
consider of particular importance for the puropses of this paper.
\begin{corollary}\label{spaces}
The spaces $\mathcal M\times\mathcal M_0$, $\mathcal M\times
\mathcal M_2$ and $\mathcal M\times\mathcal M_d$ are stable under 
the operation $\boxplus_B$ of free additive convolution of type $B$,
where
$\mathcal M_0$ is as in Definition \ref{mare}, 
$\mathcal M_2$ consists of all $\nu\in\mathcal M_0$ with the 
property that there exists a function $f\colon\mathbb R\to\mathbb R$
of bounded variation so that $\nu[(z-t)^{-1}]=\int f(t)(z-t)^{-3}\,dt$,
and $\mathcal M_d$ consists of all $\nu\in\mathcal M_2$ so that
if $\nu[(z-t)^{-1}]=\int f(t)(z-t)^{-3}\,dt$, then $f$ is 
nondecreasing.
The space $\mathcal M$ is the space of all Borel probability
measures on the real line.
\end{corollary}

\subsection{Infinitesimal freeness}
We establish next the appropriate analytic framework for the 
correspondence bewtween type $B$ free additive convolution and 
the infinitesimal freeness introduced in the previous section.
Consider a path $\gamma\colon[0,1]\to\mathcal M$. Each probability
$\gamma(t)$ has a unique nondecreasing distribution function $f_t$;
following the above notations, we write $d\gamma(t)(x)=f'_t(x)dx$.
We shall say that the path $\gamma$ is differentiable if 
$\lim_{t\to t_0}\frac{f_t-f_{t_0}}{t-t_0}$ is a (bounded) function of bounded
variation for any $t_0\in[0,1]$, where the limit is taken in the norm topology.
We have then $\int(z-x)^{-n}\,d\partial_tf_t(x)=
\partial_t\int(z-x)^{-n}\,df_t(x),$ $n\in\mathbb N$,
$\Im z\neq0$.

One can easily observe that differences of probability measures
belong to $\mathcal M_2$. Thus, it follows from the above that
$\gamma'(t)\in\mathcal M_2$. We exploit this observation in the following result, which generalizes Proposition
\ref{pro:infinitesimalbox-vs-boxplus} to distributions which need
not have moments:

\begin{theorem}\label{inf}

Assume that the functions $\gamma_j\colon[0,1]\to\mathcal M$ are
differentiable on $(0,1)$ and $\gamma_j'$ extend continuously
to $[0,1]$, $j\in\{1,2\}$. Then $(\gamma_1(t),\gamma_1'(t))
\boxplus_B(\gamma_2(t),\gamma_2'(t))=(\gamma_1(t)\boxplus
\gamma_2(t),\frac{d}{dt}
(\gamma_1(t)\boxplus\gamma_2(t)))$, for all $t\in[0,1]$.

\end{theorem}

\begin{proof}
Let us derivate in the subordination formula:

$$\partial_tG_{\gamma_1(t)\boxplus\gamma_2(t)}(z)=
G_{\gamma_j'(t)}(\omega_j^t(z))+G'_{\gamma_j(t)}(\omega_j^t(z))\partial_t
\omega_j^t(z),$$ 
for any $t\in(0,1)$, $z\in\mathbb C^+.$ (Here $\omega_j^t$ is the subordination function provided by Theorem \ref{subord} corresponding
to $\gamma_j(t).$) Part (2) of Theorem \ref{subord} implies
$$
\partial_t\omega_1^t(z)+\partial_t\omega_2^t(z)= \frac{-\partial_t
G_{\gamma_1(t)\boxplus\gamma_2(t)}(z)}{G_{\gamma_1(t)\boxplus
\gamma_2(t)}(z)^2}.
$$
 Combining these two relations gives
$$ 
\frac{\partial_tG_{\gamma_1(t)\boxplus\gamma_2(t)}(z)-
G_{\gamma_1'(t)}(\omega_1^t(z))}{G_{\gamma_1(t)}'(\omega_1^t(z))}
+\frac{\partial_tG_{\gamma_1(t)\boxplus\gamma_2(t)}(z)-
G_{\gamma_2'(t)}(\omega_2^t(z))}{G_{\gamma_2(t)}'(\omega_2^t(z))}=
\frac{-\partial_tG_{\gamma_1(t)\boxplus\gamma_2(t)}(z)}{G_{\gamma_1(t)\boxplus\gamma_2(t)}(z)^2}.
$$
We multiply the two right-hand terms by $(\omega_1^t)'(z)$
and $(\omega_2^t)'(z)$, respectively and use Theorem \ref{subord}
to get 
$$
\frac{\partial_tG_{\gamma_1(t)\boxplus\gamma_2(t)}(z)
\left[(\omega_1^t)'(z)+(\omega_2^t)'(z)\right]- 
G_{\gamma_1'(t)}(\omega_1^t(z))(\omega_1^t)'(z)-
G_{\gamma_2'(t)}(\omega_2^t(z))(\omega_2^t)'(z)}{
G_{\gamma_1(t)\boxplus\gamma_2(t)}'(z)}=
\frac{-\partial_tG_{\gamma_1(t)\boxplus\gamma_2(t)}(z)}{
G_{\gamma_1(t)\boxplus\gamma_2(t)}(z)^2}.\nonumber
$$
We multiply both sides of the equality by
$G_{\gamma_1(t)\boxplus\gamma_2(t)}'(z)$:
$${\partial_tG_{\gamma_1(t)\boxplus\gamma_2(t)}(z)
\left[(\omega_1^t)'(z)+(\omega_2^t)'(z)\right]
-G_{\gamma_1'(t})(\omega_1^t(z))(\omega_1^t)'(z)
-G_{\gamma_2'(t)}(\omega_2^t(z))(\omega_2^t)'(z)=}$$
$$\partial_t
G_{\gamma_1(t)\boxplus\gamma_2(t)}(z)
F_{\gamma_1(t)\boxplus\gamma_2(t)}'(z).$$
Formula (2) of Theorem \ref{subord} together with the equation above
assures us that
$$\partial_t
G_{\gamma_1(t)\boxplus\gamma_2(t)}(z)=
G_{\gamma_1'(t)}(\omega_1^t(z))(\omega_1^t)'(z)+
G_{\gamma_2'(t)}(\omega_2^t(z))(\omega_2^t)'(z),$$
so an application of Theorem \ref{thm17} concludes the proof.
\end{proof}

An obvious consequence of the above theorem is the following
\begin{corollary}
Let $\mu_j,\nu_j\in\mathcal M$, $j\in\{1,2\}$, and denote $\mu_j^t=t
\nu_j+(1-t)\mu_j$. Then $(\mu_1^t,\nu_1-\mu_1)\boxplus_B(\mu_2^t,\nu_2-\mu_2)=(\mu_1\boxplus\mu_2,\frac{d}{dt}
(\mu_1^t\boxplus\mu_2^t)).$

\end{corollary}

In the following two sections we shall apply Theorem \ref{inf} and the corollary above to
certain explicit type $B$ analogues of some important distributions
in free probability. 
 
\section{Some limit laws} 
We shall start with a discussion of the stable laws, first identified 
in the free context by Bercovici and Voiculescu. 
Recall \cite[Section 7]{BVIUMJ}
that two probability measures $\mu,\nu$ on the real line are said to
have the same type if there are $s>0,b\in\mathbb R$ so that
$\nu(A)=\mu(sA+b)$ for any borel set $A\subseteq\mathbb R$. We will say
that $\mu$ is stable relative to free additive convolution
if $\nu\boxplus\nu'$ has the type of $\mu$ 
whenever $\nu$ and $\nu'$ have the type of $\mu$.

We shall define the Voiculescu transform $\phi_\nu$ of the
probability measure $\nu$ by $\phi_\nu(z)=F_\nu^{-1}(z)-z$,
where $z$ belongs to a truncated Stolz angle at infinity.
Its main property, shown in \cite[Corollary 5.8]{BVIUMJ},
 is that $\phi_{\nu\boxplus\mu}(z)=
\phi_\nu(z)+\phi_\mu(z)$ for $z$ in the common domain of
the two functions\footnote{We remind the reader that 
the Voiculescu transform is related to the $R$-transform
via the formula $\phi_\nu(z)=R_\nu(1/z)$.}. It is an easy
consequence of the definition of $G$ and $\phi$ that
$G_\nu(z)=sG_\mu(sz+b)$ and
$\phi_\nu(z)=\frac{1}{s}[\phi_\mu(sz)-b]$ whenever
$\nu(A)=\mu(sA+b)$ for any Borel set $A\subseteq\mathbb R$.

In \cite[Theorem 7.5]{BVIUMJ} the authors provide a
complete list of the analytic functions on $\mathbb C^+$
which are Voiculescu transforms of stable laws relative to
free additive convolution. We recall here the list for the convenience of the reader:
\begin{enumerate}
\item $\phi(z)=a$, $a\in\mathbb R;$
\item $\phi(z)=a+ib,$ $a\in\mathbb R,b<0;$
\item $\phi(z)=a+bz^{1-\alpha}$, $a\in\mathbb R,\alpha\in(1,2]$, 
$\arg b\in[(\alpha-2)\pi,0];$
\item $\phi(z)=a+bz^{1-\alpha}$, $a\in\mathbb R,\alpha\in(0,1)$, 
$\arg b\in[\pi,(1+\alpha)\pi];$
\item $\phi(z)=a+b\log z$, $a\in\mathbb C^-\cup\mathbb R$, $b<0$.
\end{enumerate}
(The power and $\log$ functions are defined via their principal
branches; thus, $\log$
maps the upper half-plane into $\mathbb R+i(0,\pi).$)

\begin{remark}\label{20}
{\rm We observe that if we identify constants $s(t)>0,b(t)\in\mathbb R$
so that
$t\phi_\mu(z)=\phi_{\mu^{\boxplus t}}(z)= \frac{1}{s(t)}\left[\phi_\mu
(s(t)z)-b(t)\right]$, then the following formulas for $s(t)$ and $b(t)$
correspond to the five cases above:
\begin{enumerate}
\item[]Cases 1 and 2:\ $s(t)=1/t,b(t)=0$ for all $t>0$;
\item[]Cases 3 and 4:\ $s(t)=t^{-1/\alpha},b(t)=a(1-t^{1-\frac1\alpha})
$ for all $t>0$;
\item[] Case 5:\ $s(t)=1/t,b(t)=b\log t$ for all $t>0$. 
\end{enumerate}
Thus, without loss of generality, we will restrict ourselves to the 
case when $a=0$; this simply corresponds to translating our 
distributions $a$ units, or, equivalently, convolving with either
$\delta_{-a}$, in Cases 2-4, or with $\delta_{-\Re a}\boxplus\frac1\pi
\frac{-\Im a}{x^2+(\Im a)^2}dx$, in Case 5.}
\end{remark}
Stable distributions can be obtained as limits of special
triangular arrays: consider a type $A$ noncommutative probability space
$(\mathcal A,\varphi)$; if $X_1,X_2,X_3,\dots$ are free identically 
distributed random variables in $(\mathcal A,\varphi)$, if
$\displaystyle S_{n,\alpha}=\frac{X_1+X_2+\cdots+X_n-b_{n,\alpha}}{n^{\frac1\alpha}}$
and $\lim_{n\to\infty}\mu_{S_{n,\alpha}}$ exists, then its Voiculescu
transform is of one of the five forms listed above (we denote by 
$\mu_Y$ the distribution of $Y$ with respect to $\varphi$.) It is known that, 
except for the first case (corresponding to Dirac measures $\delta_a$,)
the numbers $\alpha$ from the expression of $S_{n,\alpha}$ and
from the exponent in the expression of $\phi$ are the same; thus, in
particular, $\alpha=1$, corresponding to Case 2, provides the Cauchy
distribution, and $\alpha=2$, covered by Case 3, gives us the 
semicircular law (the free central limit). 
The last case is remarkable among all others. Even
though it corresponds to $\alpha=1$, as the Cauchy distribution,
the variables $X_j$ are 'uncentered': it is impossible to obtain
a limit in this case if we try to take $b_{n,\alpha}=0$. For details,
we refer to the work of Pata \cite{Levy}.

Let us consider first Cases 3-4. We let $\mu$ be so that $\phi_\mu(z)=
bz^{1-\alpha}$. If $X_j$, $j\in\mathbb N$, are free (with respect to a
state $\varphi$,) possibly unbounded, selfadjoint
random variables distributed according to $\mu$, then we observe that
$S_{n,\alpha}^j=n^{-\frac1\alpha}(X_{nj+1}+X_{nj+2}+\cdots+X_{n(j+1)})
$, $j\in\mathbb N$, are free, identically distributed, and, as observed 
above, tend in distribution also to $\mu$ when $n$ tends to infinity. Thus, not surprisingly,
$$
S_{n,\alpha}^1+\cdots+S_{n,\alpha}^q=\frac{X_1+X_2+\cdots+X_{qn}}{n^{\frac1\alpha}}=
q^\frac1\alpha\frac{X_1+X_1+\cdots+X_r}{r^\frac1\alpha}
=q^\frac1\alpha S_{r,\alpha}.$$
Letting $n$ (and hence $r$) tend to infinity, we obtain the obvious
relation $\mu^{\boxplus q}(A)=\mu(q^\frac1\alpha A),$ for any Borel
set $A\subseteq\mathbb R$. But then,
\begin{eqnarray*}
\partial_q\int_\mathbb R f_z(x)\,d\mu^{\boxplus q}(x)=\partial_qG_{\mu^{\boxplus q}}(z) & = & \lim_{r\to\infty}
\partial_q\varphi\left([z-q^\frac1\alpha S_{r,\alpha}]^{-1}\right)\\
& = &\lim_{r\to\infty}\varphi\left([z-q^\frac1\alpha S_{r,\alpha}]^{-1}
\frac1\alpha q^{\frac1\alpha-1}S_{r,\alpha}
[z-q^\frac1\alpha S_{r,\alpha}]^{-1}
\right)\\
& = & \frac{1}{\alpha q}
\lim_{r\to\infty}\varphi\left([z-q^\frac1\alpha S_{r,\alpha}]^{-1}
q^{\frac1\alpha}S_{r,\alpha}
[z-q^\frac1\alpha S_{r,\alpha}]^{-1}
\right),\\
\end{eqnarray*}
where for any fixed $z\in\mathbb C^+$,
we have denoted $f_z(x)=(z-x)^{-1},x\in\mathbb R$. It is very easy to
see that this equality holds also in Cases 1 and 2 (with $\alpha=1$.)
It is less trivial to observe this for Case 5, and we will give a 
proof below.

Assume that, in a context similar to the one described above for Cases 1-4, the random
variables $X_j,j\in\mathbb N$ are distributed so that $\phi_\mu(z)=b
\log z$. We let $S_{n,1}^j=n^{-1}(X_{nj+1}+X_{nj+2}+\cdots+X_{n(j+1)})
-b\log n.$ Then 
\begin{eqnarray*}
S_{n,1}^1+\cdots+S_{n,1}^q
& = &\frac{X_1+X_2+\cdots+X_{qn}}{n}-qb\log n\\
& = &
q\left(\frac{X_1+X_1+\cdots+X_r}{r}-b\log r
+b\log(qn)-b\log n\right)\\
& = & q(S_{r,1}+b\log q).
\end{eqnarray*}
(We observe once again that $\phi_{\mu_{S_{n,1}}}(z)
=n^{-1}\phi_{\mu_{X_1+\cdots+X_n}}(nz)-b\log n
=\phi­_{\mu}(nz)-b\log n=b\log(nz)-b\log n=b\log z
=\phi_\mu(z),$ so the translation by $b\log n$ {\em is}
necessary.)
As in the previous four cases, with $z\in\mathbb C^+$
and $f_z(x)=(z-x)^{-1},x\in\mathbb R$,
\begin{eqnarray*}
\partial_q\int_\mathbb R f_z(x)\,d\mu^{\boxplus q}(x) & = & 
\partial_qG_{\mu^{\boxplus q}}(z)\\
& = & \lim_{r\to\infty}
\partial_q\varphi\left([z-q(S_{r,1}+b\log q)]^{-1}\right)\\
& = &\lim_{r\to\infty}\varphi\left([z-q(S_{r,1}+b\log q)]^{-1}
(S_{r,1}+b\log q+b)
[z-q(S_{r,1}+b\log q)]^{-1}
\right)\\
& = & \frac{1}{q}
\lim_{r\to\infty}\varphi\left([z-q(S_{r,1}+b\log q)]^{-1}
q(S_{r,1}+b\log q)
[z-q(S_{r,1}+b\log q)]^{-1}
\right)\\
& & \mbox{}+\frac1q\lim_{r\to\infty}\varphi\left([z-q(S_{r,1}+b\log q)]^{-1}
qb
[z-q(S_{r,1}+b\log q)]^{-1}
\right).\\
\end{eqnarray*}

Let us now consider the setup provided by Example \ref{AB}
for a type $B$ probability space. Corresponding to Cases 1-4 above,
and with the same notations, let
$$
\mathbb S_{n,\alpha}= n^{-1/\alpha}(\mathbb X_1+\cdots+\mathbb X_n)
=\left[\begin{array}{cc}
\frac{X_1+\cdots+X_n}{n^\frac1\alpha} & \frac{1}{\alpha}\frac{X_1+\cdots+X_n}{n^{\frac1\alpha}}\\
0 & \frac{X_1+\cdots+X_n}{n^\frac1\alpha}
\end{array}\right]
=\left[\begin{array}{cc}
S_{n,\alpha} & \frac{1}{\alpha}S_{n,\alpha}\\
0 & S_{n,\alpha}
\end{array}\right],$$ let $(\mu,\nu)$ be the type $B$ distribution
of $\mathbb X_j$, $j\in\mathbb N$, and denote $(\mu_q,\nu_q)=
(\mu,\nu)^{\boxplus q}$.
Thus, 
$$n^{-1/\alpha}(\mathbb X_1+\cdots+\mathbb X_{qn})=
\left[\begin{array}{cc}
q^\frac1\alpha\frac{X_1+\cdots+X_{qn}}{(qn)^\frac1\alpha} & 
q^{\frac1\alpha}\frac{X_1+\cdots+X_{qn}}{\alpha(qn)^\frac1\alpha}\\
0 & q^\frac{1}{\alpha}\frac{X_1+\cdots+X_{qn}}{(qn)^\frac1\alpha}
\end{array}\right]=q^{1/\alpha}r^{-1/\alpha}(\mathbb X_1+\cdots+\mathbb X_r),$$
and passing to the limit when $n\to\infty$ provides, together with
Example \ref{AB}, stability for the type $B$ distribution $(\mu,\nu)$.
We will prove that $g_{\nu}(z)=\partial_q|_{q=1}G_{\mu_q}(z)=\partial_q|_{q=1}
G_{\mu^{\boxplus q}}(z)$, where we remind the reader that the lower
case $g$ refers to the Cauchy transform of a distribution
corresponding to a second coordinate in a type $B$ probability space.
Indeed, let us write, for $Z\in\mathcal C$,
\begin{eqnarray*}
\left[Z-q^\frac1\alpha\mathbb S_{r,\alpha}\right]^{-1} & = &
\left[\begin{array}{cc}
z-q^\frac1\alpha S_{r,\alpha} & w-\frac1\alpha q^\frac1\alpha S_{r,\alpha}\\
0& z-q^\frac1\alpha S_{r,\alpha}
\end{array}\right]^{-1}\\
& = & \left[\begin{array}{cc}
(z-q^\frac{1}{\alpha}S_{r,\alpha})^{-1} & (z-q^\frac1\alpha 
S_{r,\alpha})^{-1}(w-\frac1\alpha q^\frac1\alpha S_{r,\alpha})
(z-q^\frac1\alpha S_{r,\alpha})^{-1}\\
0& (z-q^\frac{1}{\alpha}S_{n,\alpha})^{-1}
\end{array}\right]
\end{eqnarray*}
Applying $(\varphi,\varphi)$ on the above, we obtain
$$
(G_{\mu_q}(z),wG_{\mu_q}'(z)+g_{\nu_q}(z))
=(G_{\mu_q}(z),wG_{\mu_q}'(z)+q\partial_qG_{\mu^{\boxplus q}}
(z));$$
evaluating this relation in $q=1$ of course proves our claim.

We focus next on Case 5. For this we shall let
$$
\mathbb S_{n,1}= n^{-1}(\mathbb X_1+\cdots+\mathbb X_n)
=\left[\begin{array}{cc}
\frac{X_1+\cdots+X_n}{n}-b\log n & b+\frac{X_1+\cdots+X_n}{n}-b\log n\\
0 & \frac{X_1+\cdots+X_n}{n}-b\log n
\end{array}\right]
=\left[\begin{array}{cc}
S_{n,1} & b+S_{n,1}\\
0 & S_{n,1}
\end{array}\right],$$ 
(thus, on the second coordinate we shift each $X_j$ with a $b$) and 
leave the rest of notations/conventions as before. Then 
\begin{eqnarray}
\lefteqn{n^{-1}(\mathbb X_1+\cdots+\mathbb X_{qn}) =}\nonumber\\
& & \mbox{}
\left[\begin{array}{cc}
q\left(\frac{X_1+\cdots+X_{qn}}{qn}-b\log(qn)+b\log q\right) & 
q\left(b+\frac{X_1+\cdots+X_{qn}}{qn}-b\log(qn)+b\log q\right)\\
0 & q\left(\frac{X_1+\cdots+X_{qn}}{qn}-b\log(qn)+b\log q\right)
\end{array}\right]=\nonumber\\
& & \mbox{} qr^{-1}(\mathbb X_1+\cdots+\mathbb X_r),\nonumber
\end{eqnarray}
Obviously, as for the first four cases, we have
$$g_{\nu_q}(z)=
q\partial_qG_{\mu^{\boxplus q}}(z).$$
Evaluating in $q=1$ provides the desired result.

It is worth mentioning that $\partial_q G_{\mu^{\boxplus q}}(z)$ can 
be expressed in terms of $G_{\mu^{\boxplus q}}(z)$ and its derivative
with respect to $z$: for Cases 1-4, we have
\begin{eqnarray*}
\partial_qG_{\mu^{\boxplus q}}(z)
& = & \frac{1}{q\alpha}\int_\mathbb R\frac{x}{(z-x)^2}\,d\mu^{\boxplus q}(x)\\
& = & \frac{1}{q\alpha}\int_\mathbb R xf_z'(x)\,d\mu^{\boxplus q}(x),\\
& = & -\frac{1}{q\alpha}(G_{\mu^{\boxplus q}}(z)+zG_{\mu^{\boxplus q}}'
(z)),
\end{eqnarray*}
while for Case 5,
\begin{eqnarray*}
\partial_qG_{\mu^{\boxplus q}}(z)
& = & \frac{1}{q}\int_\mathbb R\frac{x}{(z-x)^2}\,d\mu^{\boxplus q}(x)
+\int_\mathbb R\frac{b}{(z-x)^2}\,d\mu^{\boxplus q}(x)\\
& = & \frac{1}{q}\int_\mathbb R xf_z'(x)\,d\mu^{\boxplus q}(x)-
bG_{\mu^{\boxplus q}}'(z)\\
& = & -\frac{1}{q}(G_{\mu^{\boxplus q}}(z)+zG_{\mu^{\boxplus q}}'
(z))-bG_{\mu^{\boxplus q}}'(z).
\end{eqnarray*}
These two formulae guarantee us that $(\mu^{\boxplus q},
\partial_q\mu^{\boxplus q})\in\mathcal M\times\mathcal M_2.$
(In Case 3, $\alpha=2$, $G_{\mu^{\boxplus q}}(z)$ satisfies a
more famous equation, the complex Burgers equation - see
\eqref{Burgers} below, with $\mu$ from equation
\eqref{Burgers} taken to be $\delta_0$.)

\begin{remark}\label{surprise}
{\rm The operation $\boxplus_B$ behaves well with respect to 
translations. Indeed, let us consider $(\mu,\sigma),
(\nu,\rho)\in\mathcal M\times
\mathcal M_2.$ Then a simple computation using Theorem \ref{subord}
shows that for given translation $(\mu^b,\sigma^b)=
(\mu(\cdot-b),\sigma(\cdot-b))$ we have $G_{\mu^b\boxplus\nu^c}(z)=
G_{\mu\boxplus\nu}(z+b+c)$, so the subordination functions satisfy
$\omega_1^{bc}(z)=\omega_1(z+b+c)-b$, $\omega_2^{bc}(z)=\omega_2
(z+b+c)-c$, where $G_{\mu\boxplus\nu}(z)=G_\mu(\omega_1(z)),
G_{\mu^b\boxplus\nu^c}(z)=G_{\mu^b}(\omega_2^{bc}(z))$, and similarly
for $\omega_2,\nu$ and $c$. So, according to Theorem \ref{thm17},
$g_{\sigma^b}(\omega_1^{bc}(z))
(\omega_1^{bc})'(z)+g_{\rho^b}(\omega_2^{bc}(z))
(\omega_2^{bc})'(z)=g_\sigma(\omega_1(z+b+c)-b+b)
\omega_1'(z+b+c)+g_\rho(\omega_2(z+b+c)-c+c)
\omega_2'(z+b+c)=g_{\lambda}(z+b+c)=g_{\lambda^{b+c}}(z),$
where we denote $(\mu,\sigma)\boxplus_B(\nu,\rho)=(\mu\boxplus\nu,\lambda)$. Thus, $\boxplus_B$ behaves well with respect to simple
translations.

}
\end{remark}
We can now prove the following corollary of Theorem \ref{inf}:
\begin{corollary}\label{cor29}
Let $(\mathcal A,\varphi)$ be a $*$-probability space of type $A$, and 
$X=X^*\in\mathcal A$ 
be 
selfadjoint random variable whose 
distribution $\mu$ is stable with respect to free additive 
convolution. Let $(\mathcal A,\varphi,\mathcal A,\varphi)$ be the 
type $B$ noncommutative probability space obtained from $(\mathcal A,
\varphi)$ as in Example \ref{AB}. Then there exists $b\in\mathbb R$
so that $(\mu,\nu)^{\boxplus_B q}=(\mu^{\boxplus q},
\partial_t|_{t=1}(\mu^{\boxplus tq})),$ where 
$(\mu,\nu)$ denotes the distribution of $(X,X+b)$ with
respect to $(\varphi,\varphi)$.
\end{corollary}
\noindent{\bf Proof:} The result follows from the above considerations,
Remark \ref{surprise} and Theorem \ref{inf}.
\bigskip

\begin{example}\label{24}
{\rm The above results allow us to recover Popa's Central
Limit \cite{MVP}. Indeed, the Cauchy transform of the centered
semicircular law $d\gamma_t(x)=(2t^2\pi)^{-1}\sqrt{4t-x^2}dx$
of variance $t$ is $G_{\gamma_t}(z)=(z-\sqrt{z^2-4t})/2t$.
Differentiating with respect to $t$ gives
$$
\partial_tG_{\gamma_t}(z)=\frac1t\left(\frac{1}{\sqrt{z^2-4t}}-
\frac{z-\sqrt{z^2-4t}}{2t}\right),
$$
which is exactly $t^{-1}$ times the second coordinate of
the type $B$ free central limit. It is easy to observe from
this expression that $t\partial_t\gamma_t$ is nothing more than
the difference of an arcsine law and a semicircular distribution.
This corresponds to Case 3, $\alpha=2$, among the stable distributions.

In the same way, one may obtain the type $B$ corresopndent
of the Cauchy distribution (Case 2): if $c_t$ is the 'centered' Cauchy 
distribution
$dc_t(x)=\frac{1}{\pi}\frac{t}{x^2+t^2}$, $t>0$, then $G_{c_t}(z)=
\frac{1}{z+it}$, so that 
$$
\partial_tG_{c_t}(z)=\frac{-it}{(z+it)^2},$$
and the corresponding measure is $\pi^{-1}\frac{xt^2}{(x^2+t^2)^2}dx.$

Case 1 (Dirac measures) is simple: the second coordinate is simply
the distribution $d$ which acts on any smooth function on $\mathbb R$
as $d(f)=f'(a)$. This is obviously not a measure.

}
\end{example}
The reader will notice that Corollary \ref{cor29} does not claim to
describe all $B$-freely stable laws. Indeed, due to the nature of 
type $B$ random variables, it is not fully clear what a 
$B$-freely stable law should be. To start with, if $(X,\xi)$ is
a type $B$ random variable, then 
$n^{-1}[(X,\xi)-(b,c)]$ has a Cauchy transform given on coordinates
by
$$\left(n G_{X}(nz+b),w[n G_{X}(nz+b)]'+ng_\xi(nz+b)+\frac{c}{n}
[nG_{X}(nz+b)]'\right),$$
where the reader will notice that we have already implicitly
assumed to be in the context provided by Example \ref{AB}.
Indeed, taking the expectation of 
$\left(Z-n^{-1}[(X,\xi)-(b,c)]\right)^{-1}$ gives
\begin{eqnarray*}
E\left(\left[\begin{array}{cc}
z-\frac{X-b}{n} & w-\frac{\xi-c}{n}\\
0 & z-\frac{X-b}{n}
\end{array}\right]^{-1}\right) & = & 
E\left[\begin{array}{cc}
\left(z-\frac{X-b}{n}\right)^{-1} & 
-\left(z-\frac{X-b}{n}\right)^{-1}\left(w-\frac{\xi-c}{n}
\right)\left(z-\frac{X-b}{n}\right)^{-1}\\
0 & \left(z-\frac{X-b}{n}\right)^{-1} 
\end{array}\right]\\
& = & \left[\begin{array}{cc}
\varphi\left[\left(z-\frac{X-b}{n}\right)^{-1}\right] & 
-\left(w+\frac{c}{n}\right)\varphi\left[\left(z-\frac{X-b}{n}\right)^{-2}\right]\\
0 & \varphi\left[\left(z-\frac{X-b}{n}\right)^{-1}\right] 
\end{array}\right]\\
& & \mbox{}+\left[\begin{array}{cc}
0 & 
\frac{1}{n}\varphi\left[\left(z-\frac{X-b}{n}\right)^{-1}\xi\left(z-\frac{X-b}{n}\right)^{-1}\right]\\
0 & 0
\end{array}\right].\\
\end{eqnarray*}
While the first coordinate does not raise any issues, it is far
from clear whether its interaction with the second should be allowed
to include the addition of $\frac{c}{n}
[nG_{X}(nz+b)]',$ when $c\neq b$, for example.

A way to answer this question would be to follow the one-variable
results and consider limits of sums
$$(S_{n,\alpha},s_{n,\alpha})=\frac{(X_1,\xi_1)+(X_2,\xi_2)+\cdots+
(X_n,\xi_n)-(b_{n,\alpha},c_{n,\alpha})}{n^\frac1\alpha},$$
where $(X_1,\xi_1),(X_2,\xi_2),(X_3,\xi_3),\dots$ are $B$-free
identically distributed in some type $B$ probability space 
$(\mathcal A,\varphi,\mathcal A,\varphi)$, with selfadjoint components,  
and the distribution belongs to $\mathcal M\times\mathcal M_2.$ Thus, 
we implicitly require that our probability space is of the form 
provided by Example \ref{AB}. It is  clear that the first component 
tends to one of the freely stable laws described by Bercovici and 
Voiculescu. The second coordinate has a Cauchy transform given
by
$$
\varphi\left[\left(z-\frac{
X_1+\cdots+X_n-b_{n,\alpha}}{n^\frac1\alpha}\right)^{-1}
\frac{\xi_1+\cdots+\xi_n-c_{n,\alpha}}{n^\frac1\alpha}\left(z-\frac{
X_1+\cdots+X_n-b_{n,\alpha}}{n^\frac1\alpha}\right)^{-1}\right].
$$
One can observe that, provided 
$\frac{\xi_1+\cdots+\xi_n-c_{n,\alpha}}{n^\frac1\alpha}$ and 
$\frac{
X_1+\cdots+X_n-b_{n,\alpha}}{n^\frac1\alpha}$ converge in distribution,
our choice in the previous corollary corresponds to having $X_j$ and
$\xi_j$ belong to the same domain of attraction. Under these 
circumstances, one can consider the problem of the stable laws 
answered. However, it is a different,
and considerably more complicated, matter when these two variables
belong to different domains of attraction, or especially when 
$\frac{\xi_1+\cdots+\xi_n-c_{n,\alpha}}{n^\frac1\alpha}$ does
{\em not} converge in distribution, and we will not attempt to
solve these problems here (the correspondence invoked in
Remark \ref{rmk16} part 2 would be useful provided that the nature
of the limit laws, not only their existence, could be identified.)


We discuss next in more detail several aspects of Popa's central limit
theorem.
\subsection{The type $B$ analogue of the heat equation.}
It was shown in \cite{fe1} that if $\mu\in\mathcal M$, then
the type $A$ free analogue of the heat equation is the complex Burgers
equation:
\begin{equation}\label{Burgers}
\partial_tG_{\mu\boxplus\gamma_t}(z)+G_{\mu\boxplus\gamma_t}(z)
\partial_zG_{\mu\boxplus\gamma_t}(z)=0,\quad z\in\mathbb C^+,t>0.
\end{equation}
For distributions of type $B$, one can speak a priori of two versions
of the heat equation, depending on whether one considers the parameter
$t$ to be a positive number (identified in $\mathcal C$ with the matrix
having $t$ on the diagonal and zero elsewhere), or one takes 
${\bf t}= \left[\begin{array}{cc}
t & s\\
0 & t
\end{array}\right]$ with $t>0,s\in\mathbb R$ instead. We shall see 
that in fact (probably not surprisingly) there is no difference between
these two versions. To prove this statement, we shall consider the 
second case, and show it reduces to the first. Using the essential 
observation of Biane, Goodman and Nica that 
{``}type $B$ freeness = type $A$ freeness over $\mathcal C$''
(see Corollary \ref{functionalequation} and remarks following it) and 
analyticity
of the correspondences $t\mapsto G_{\mu\boxplus\gamma_t}(z)$
and $z\mapsto G_{\mu\boxplus\gamma_t}(z)$, it is easy to observe 
that equation \eqref{Burgers} holds when we replace $z$ with 
$Z=\left[\begin{array}{cc}
z & w\\
0 & z
\end{array}\right]\in\mathcal C,$ $z\in\mathbb C^+$, and $t$ with
${\bf t}$ from above.
The meaning of $\partial_Z$ is quite clear when one uses the 
power series formalism: if, as in introduction, $f(Z)=
\sum_{n=0}^\infty A_nZ^n=(f_1(z),wf_1'(z)+f_2(z))$, $A_n
=\left[\begin{array}{cc}
a_n & b_n\\
0 & a_n
\end{array}\right],$ $f_1(z)=\sum_{n=0}^\infty a_nz^n$, $f_2(z)=
\sum_{n=0}^\infty b_nz^n$, then on components 
$$\partial_Z f(Z)= \sum_{n=1}^\infty nA_nZ^{n-1}=(f_1'(z),wf_1''(z)+
f_2'(z)).$$
For functions (maps) $f(Z,{\bf t})\colon\mathcal C\times\mathcal C
\to\mathcal C$, we have then
\begin{eqnarray*}
f(Z,{\bf t}) & = & \sum_{m,n=0}^\infty A_{m,n}Z^{m}{\bf t}^n\\
& = & \sum_{m,n=0}^\infty \left[\begin{array}{cc}
a_{m,n}z^mt^n & sna_{m,n}z^mt^{n-1}+wma_{m,n}z^{m-1}t^n+b_{m,n}z^mt^n\\
0 & a_{m,n}z^mt^n
\end{array}\right]\\
& = & \left[\begin{array}{cc}
f_1(z,t) & s\partial_tf_1(z,t)+w\partial_zf_1(z,t)+f_2(z,t)\\
0 & f_1(z,t)
\end{array}\right],
\end{eqnarray*}
where obviously $f_1(z,t) =\sum_{m,n=0}^\infty a_{m,n}z^mt^n $
and $f_2(z,w)=\sum_{m,n=0}^\infty b_{m,n}z^mt^n.$ It is clear
that $\partial_Zf(Z,{\bf t})=(\partial_zf_1(z,t),
s\partial_z\partial_tf_1(z,t)+w\partial_z^2f_1(z,t)+\partial_zf_2(z,t))
$ and similar for ${\bf t}$. Thus writing the complex Burgers equation
for $f$ in $Z$ and {\bf t} on components gives
the usual complex Burgers equation for $f_1(z,t)$, while
for the second one obtains
$$s(\underbrace{\partial_t^2f_1+\partial_t[f_1\partial_zf_1]}_{\mathfrak B_1})+
w(\underbrace{\partial_z\partial_tf_1+\partial_z[f_1\partial_zf_1]}_{\mathfrak B_2})+
\underbrace{\partial_tf_2+\partial_z(f_1f_2)}_{\mathfrak Q}=0.$$
Here both $f_1,f_2$ are functions of two variables $(z,t)$.
As $f_1$ satisfies the complex Burgers equation, it follows
trivially that $\mathfrak B_1=\mathfrak B_2=0$.
Thus, the only nontrivial component is $\mathfrak Q$. The complex
Burgers equation on $\mathcal C$ should thus be written as
\begin{equation}\label{mathcalC}
\left\{\begin{array}{rcl}
\partial_tf_1(z,t)+f_1(z,t)\partial_zf_1(z,t) & = & 0\\
\partial_tf_2(z,t)+\partial_z[f_1(z,t)f_2(z,t)] & = & 0
\end{array}\right.
\end{equation}
What is remarkable is that the second coordinates of the 
variables do not appear at all in the equation. Thus, assume
$(\gamma_t,\lambda_t),t\ge0$ is the one-parameter semigroup of
the type $B$ free central limit law, and $ (\mu,\sigma) \in\mathcal M
\times\mathcal M_2$ is fixed. If we denote $(\gamma_t,\lambda_t)
\boxplus_B(\mu,\sigma)=(\mathfrak G(t),\mathfrak L(t))$, then the 
type $B$ free analogue of the heat equation is given by
\begin{equation}\label{heattypeB}
\left\{\begin{array}{rcl}
\partial_t G_{\mathfrak G(t)}(z)+G_{\mathfrak G(t)}(z)
\partial_zG_{\mathfrak G(t)}(z) & = & 0\\
\partial_tg_{\mathfrak L(t)}(z)+\partial_z[G_{\mathfrak G(t)}(z)
g_{\mathfrak L(t)}(z)] & = & 0
\end{array}\right.
\end{equation}
with initial conditions $$G_{\mathfrak G(0)}(z)\ = \ G_\mu(z),\quad
g_{\mathfrak L(0)}(z)\ = \ g_\sigma(z).$$
It might be of interest to record the fact that if
$P$ is a one-variable analytic function defined on a neighbourhood
of the closed upper half-plane in $\mathbb C$, then
$$\partial_t[g_{\mathfrak L(t)}(z)P(G_{\mathfrak G(t)}(z))]
+\partial_z[g_{\mathfrak L(t)}(z)G_{\mathfrak G(t)}(z)
P(G_{\mathfrak G(t)}(z))]=0.$$

\section{Some families of laws and the type B central limit theorem}

As observed above,
the (single variable) central limit theorem for infinitesimal free
convolution is equivalent to the central limit theorem for type $B$
free convolution. Such a theorem was obtained by M. Popa \cite{MVP}.
The associated {}``infinitesimal semicircle law'' is one for which
$\mu$ is the semicircule measure and $\mu'(t^{n})=0$ if $n$ is
odd and $\mu'(t^{2k})=2k\mu(t^{2k})$ (the moments of the difference of
the arcsine and semicircular laws).

\subsection{Combinatorial interpretation of type $B$ semicircle law.}

As for the type $A$ central limit distribution, a particular case of 
interest in this context among noncrossing partitions of type $B$ is a noncrossing
pairing:
\begin{definition}
A \emph{type $B$ non-crossing pairing} of size $n$ is a type $B$ non-crossing
partition $\pi$ of $\{1,\ldots,n,-1,\ldots,-n\}$, so that either
(i) all blocks of $\pi$, except for the zero block, consist of two
elements and (ii) either $\pi$ has no zero block, or its zero block
has the form $\{i,j,-i,-j\}$.
\end{definition}
Clearly, $n$ has to be even for such a non-crossing pairing to exist.
In either case, the absolute value of $\pi$ is a non-crossing pairing
(of type $A$). In particular, exactly one of $i,j$ must be even. 

\begin{lemma}
\label{lem:abs-is-a-cover}The map $\pi\mapsto Abs(\pi)$ is a $(k+1)$-to-one
cover of the set of type $A$ non-crossing pairings of $\{1,\ldots,2k\}$
by type $B$ non-crossing pairings of size $2k$. More precisely, given
a type $A$ non-crossing pairing $\pi'$ and a set $K$ which is either
empty or is a block of $\pi'$, there exists a unique non-crossing
pairing $\rho(\pi,K)$ of type $B$, with $Abs(\pi)=\pi'$ and so
that the zero block of $\pi$ given by $\{\pm i:i\in K\}$.
\end{lemma}
\begin{proof}
Let $\pi'$ be a fixed type $A$ non-crossing pairing of $\{1,\ldots,2k\}$.
Choose a block $\{i,j\}$ of $\pi'$, and let's assume that $i<j$.
Now consider $\pi$ defined on $\{1,\ldots,n,-1,\ldots,-n\}$ as follows.
First, $\{i,j,-i,-j\}$ is a block of $\pi$. Next, if $\{p,q\}$
is a block of $\pi'$ with $i<p<q<j$, then both $\{p,q\}$ and $\{-p,-q\}$
are blocks of $\pi$. If $\{p,q\}$ is a block of $\pi'$ with $p<q$,
and either $q<i$, or $p>j$ or $p<i<j<q$, then $\{p,-q\}$ and $\{-p,q\}$
are both blocks of $\pi$. Alternatively, $\pi=\rho(\pi',K)$ can
be described by insisting that $\pi|_{\{i+1,\ldots,j-1\}}=\pi'|_{\{i+1,\ldots,j-1\}}$,
$\pi|_{\{j+1,\dots,n,-1,\dots,-(i-1)\}}=\pi'|_{\{j+1,\dots,n,1,\dots,i-1\}}$
(which is non-crossing, since we have applied a cyclic permutation
of the set underlying $\pi'$) and by the condition that blocks of
$\pi$ are preserved by inversion. It is clear from this description
that $\pi$ is non-crossing. Moreover, it is clear from this description
that this is the unique B non-crossing parition $\pi$ with zero block
$\{i,j,-i,-j\}$ and absolute value $\pi'$.

Since there are $k$ choices of a block of $\pi'$, we have constructed
$k$ type $B$ pair partitions with absolute value $\pi'$ (and all having
a specified zero block). We can construct one more type $B$ partition,
by stating that $\{i,j\},\{-i,-j\}$ are both blocks of $\pi$ whenever
$\{i,j\}$ is a block of $\pi'$. This partition has absolute value
$\pi$ and no zero block, and hence is the unique type $B$ non-crossing
pairing with this property. To conclude the proof, we set $\rho(\pi',\emptyset)=\pi$.
\end{proof}
Let $C_{n}$ be the number of non-crossing pairings of $\{1,\dots,2n\}$
(thus $C_{n}$ is the Catalan number $\frac{1}{n+1}\binom{2n}{n}$).
Let $B_{n}$ be the number of non-crossing pairings of $\{1,\dots,2n\}$
having a nontrivial zero block. Then we see that $B_{n}=nC_{n}$ by
Lemma \ref{lem:abs-is-a-cover}. Thus we obtain a combinatorial interpretation
for the moments of a type $B$ semicircular law:

\begin{prop}
Let $(\mu,\mu')$ be the infinitesimal law of a type $B$ semicircular
random variable. Then for all $k=0,1,\dots$, $\mu(t^{2k})=C_{k}$,
$\mu'(t^{2k})=B_{k}$ and $\mu(t^{2k+1})=\mu'(t^{2k+1})=0$.
\end{prop}

\subsection{Families of type $B$ semicircular variables.}

Let $C(i_{1},\dots,i_{k})$ be the number of all non-crossing pairings
of $\{1,\dots,k\}$ for which $i_{p}=i_{q}$ whenever $p\stackrel{\pi}{\sim}q$
(i.e. these are color-preserving partitions of $\{1,\dots,k\}$ in
which the $p$-th digit is colored by the color $i_{p}$). Let $B(i_{1},\dots,i_{k};j)$
be the number of type $B$ non-crossing pairings of $\{1,\dots,k\}$
for which $i_{p}=i_{q}$ whenever $|p|\stackrel{\pi}{\sim}|q|$ and
for which the zero block contains $j$. Note that $B(i_{1},\dots,i_{k};j)=C(i_{1},\dots,i_{k})$,
since $Abs(\pi)$ together with the designation of which pair in $Abs(\pi)$
comes from the zero block of $\pi$ determines $\pi$ fully, and since
the parity of $j$ determines uniquely whether it can be the smallest
or largest element in a class of $Abs(\pi)$.

By analogy with the single-variable case, we shall call a family of
type $B$ non-commutative random variables $((X_{1},\xi_{1}),\ldots,(X_{n},\xi_{n}))$
in a type $B$ probability space $(A,\mathcal{V},\tau,f,\Phi)$ a type
B semicircular family if its law is given by\begin{eqnarray*}
\tau(X_{i_{1}}\cdots X_{i_{k}}) & = & C(i_{1},\dots,i_{k})\\
f(X_{i_{1}}\cdots X_{i_{j-1}}\xi_{i_{j}}X_{i_{j+1}}\cdots X_{i_{k}}) & = & B(i_{1},\dots,i_{k},;j).\end{eqnarray*}
In particular, note that the variables $(X_{1},\ldots,X_{n})$ form
a free semicircular family.

\begin{lemma}
Let $((X_{1},\xi_{1}),\ldots,(X_{n},\xi_{n}))$ be type $B$ non-commutative
random variables as above. Then they are (type $B$) freely independent.
\end{lemma}
\begin{proof}
Since the joint (type $A$) law of $(X_{1},\ldots,X_{n})$ is that of
a semicircular family, it follows that $X_{1},\ldots,X_{n}$ are (type
A) freely independent. 

Fix $i_{1},\dots,i_{k}$. For $\pi$ a non-crossing pairing of type
A, let $c(\pi)=1$ if $i_{j}=i_{j'}$ whenever $j\stackrel{\pi}{\sim}j'$
and $c(\pi)=0$ otherwise. Then we may write, in view of Lemma \ref{lem:abs-is-a-cover}\begin{eqnarray*}
f(X_{i_{1}}\cdots X_{i_{j-1}}\xi_{i_{j}}X_{i_{j+1}}\cdots X_{i_{k}}) & = & \sum_{{\pi\textrm{ type $B$ non-crossing}\atop \textrm{having a zero block sarting at }j}}c(Abs(\pi))\\
 & = & \sum_{Abs(\pi)}c(Abs(\pi))\\
 & = & \sum_{\rho\textrm{ type $A$ non-crossing}}\prod_{\{a,b\}\textrm{ class of }\rho}\delta_{i_{a}=i_{b}}.\end{eqnarray*}
Thus if we set $\kappa(X_{1},\ldots,\xi_{j},\ldots,X_{p})=0$ unless
$p=2$ and let $\kappa(X_{i},X_{j})=\kappa(\xi_{i},X_{j})=\kappa(X_{i},\xi_{j})=\delta_{i=j}$,
then we have\[
f(X_{i_{1}}\cdots X_{i_{j-1}}\xi_{i_{j}}X_{i_{j+1}}\cdots X_{i_{k}})=\sum_{\rho\textrm{ type $A$ non-crossing}}\kappa_{\rho}(X_{i_{1}},\ldots,\xi_{i_{j}},\ldots,X_{i_{k}}).\]
Since $X_{1},\ldots,X_{n}$ are (type $A$) semicircular and free we
also have a similar formula for $\tau$:\[
\tau(X_{i_{1}}\cdots X_{i_{k}})=\sum_{\rho\textrm{ type $A$ non-crossing }}\kappa_{\rho}(X_{i_{1}},\dots,X_{i_{k}}).\]
It follows from formula (6.13) on p. 2292 of \cite{BGN} that the $(A')$
cumulants of the type $B$ family $X_{1},\ldots,X_{n}$ are exactly the
functionals $\kappa$ that we defined above. Since $\kappa$ obviously
satisfies the condition that mixed cumulants vanish, it follows that
our family is indeed type $B$ freely independent (see Theorem 6.4 on p. 
2293, Proposition on p. 2298 and Corollary on p. 2300 of \cite{BGN}).
\end{proof}

\subsection{The type $B$ Poisson limit theorem}
We consider next the issue of the Poisson limit theorem for type
$B$ distributions.
Recall from the paper of M. Popa \cite{MVP}
that a type $B$ Bernoulli variable
has the type $B$ law given by\[
E((a,\xi)^{n})=\Lambda
A^{n}\]
where 
$\Lambda= \left[\begin{array}{cc}
\lambda_1 & \lambda_2\\ 
0& \lambda_1
\end{array}\right]$
and $A=\left[\begin{array}{cc}
\alpha_1 & \alpha_2\\ 
0& \alpha_1
\end{array}\right]$
. The formula is valid for all
$n\geq1$ (for $n=0$ the correct answer is 
the identity in $\mathcal C$). Thus explicitely
we have\[
E((a,\xi)^{n})=\left[\begin{array}{cc}
\lambda_1\alpha_1^n & \lambda_2\alpha_1^n+n\lambda_1\alpha_1^{n-1}\alpha_2\\ 
0& \lambda_1\alpha_1^n
\end{array}\right]
.\]
Let now\[
X_{t}=\alpha_{1}\exp\left(\frac{\alpha_{2}}{\alpha_{1}}t\right)p_{t},\]
where $p_{t}$ is a family of projections so that $\tau(p_{t})=\lambda_{1}+t\lambda_{2}$.%
\footnote{Of course, we could modify $X_{t}$ by removing higher order terms
in $t$, so for example we could use $X_{t}=\alpha_{1}(1+t\alpha_{2}/\alpha_{1})p_{t}$.%
}Then:\begin{eqnarray*}
\tau(X_{t}^{n}) & = & \alpha_{1}^{n}\exp\left(\frac{\alpha_{2}}{\alpha_{1}}nt\right)\tau(p_{t})=\alpha_{1}^{n}\exp\left(\frac{\alpha_{2}}{\alpha_{1}}nt\right)(\lambda_{1}+t\lambda_{2})\\
 & = & \lambda_{1}\alpha_{1}^{n}+t(\alpha_{1}^{n}\frac{\alpha_{2}}{\alpha_{1}}n\lambda_{1}+\alpha_{1}^{n}\lambda_{2})+O(t^{2})\\
 & = & \lambda_{1}\alpha_{1}^{n}+t(\lambda_{2}\alpha_{1}^{n}+n\lambda_{1}\alpha_{1}^{n-1}\alpha_{2})+O(t^{2}),\end{eqnarray*}
so that $X_{t}$ has the desired infinitesimal law.

Thus, one can re-interpret Popa's type $B$ Poisson summation theorem as
essentially a direct consequence of the free Poisson summation theorem,
since for each fixed $t$, appropriately rescaled free sums of $X_{t}^{n}$
will have a free Poisson law with parameters $(\alpha_{1}\exp(t\alpha_{2}/\alpha_{1}),\lambda_{1}+t\lambda_{2})$
as a limit. (We remind the reader that a free Poisson law with 
parameters $(\alpha,\lambda)$ is 
the result of applying the homotethy by $\alpha$ to the 
law with with $R$-transform $\lambda z(1-z)^{-1}$).

We state our result in the following corollary:
\begin{corollary}
Free type $B$ Poisson laws are infinitesimals
to the family of free Poisson laws with the parameters $(\alpha_{1}\exp(t\alpha_{2}/\alpha_{1}),\lambda_{1}+t\lambda_{2})$.
An operator model for these laws is 
\[
Y_{t}=\alpha_{1}\exp(t\alpha_{2}/\alpha_{1})Xp_{t}X=X\ B_{t}X\]
where $X$ is a standard semicircular and $p_{t}$ are projections
with $\tau(p_{t})=\lambda_{1}+t\lambda_{2}$, and $B_{t}=\alpha_{1}\exp(t\alpha_{2}/\alpha_{1})$
is type $B$ Bernoulli. Moreover, any product of the form
\[
Z_{t}\ B_{t}\ Z_{t}\]
where $Z_{t}$ is a type $B$ semicircular and $B_{t}$ a type $B$ Bernoulli, 
is a type $B$ Poisson.
\end{corollary}

In the spirit of the above corollary, we recall that, as shown by
Nica and Speicher in \cite{NSmult},
for an arbitrary probability measure on the real line
$\mu$, one can define the $t^{\rm th}$ free additive
convolution power $\mu^{\boxplus t}$ for any $t\ge1$;
moreover, \cite{NSmult} provides also an operatorial 
representation of $\mu^{\boxplus t}$. If $X=X^*\in
(\mathcal A,\varphi)$ is distributed according to $\mu$,
and $p(t)=p(t)^2=p(t)^*\in(\mathcal A,\varphi)$ is a
projection free from $X$ with $\varphi(p(t))=t^{-1}$, then 
$tp(t)Xp(t)\in(p(t)\mathcal Ap(t),\frac1t\varphi)$ is distributed
according to $\mu^{\boxplus t}$. Thus, in a certain sense,
derivating along the path $t\mapsto\mu^{\boxplus t}$ with respect
to $t$ corresponds at an operatorial level to derivating along
a path of free projectors. This observation holds in particular
for the stable laws described in Corollary \ref{cor29}.

\section{A matrix model for type $B$ freeness.}

The main result of this section deals with the asymptotics of the
law of a matrix $X_{N}$ whose entries are semicircular variables.
More precisely, we assume that $X_{N}=X_{N}^{*}$ is $N\times N$
and that its entries $\{X_{N}^{ij}:1\leq i,j\leq N\}$ form a free
semicircular family so that the covariance of $X_{N}^{ij}$ is $N^{-1/2}(1+\delta_{i=j})$.
The matrix $X_{N}$ is a free probability analog of a real Gaussian
random matrix (the free probability analog of a complex Gaussian random
matrix differs in that the off-diagonal entries $X_{N}^{ij}$, $i<j$,
are circular rather than semicircular). Note in particular that $X_{N}=X_{N}^{t}$
if by the latter we denote the a kind of transpose of $X_{N}$ obained
by switching its rows and columns.

Although $X_{N}$ has in the $N\to\infty$ limit a semicircular law
of variance $1$, its law varies with $N$ (in fact, we'll show that
the law of $X_{N}$ is seicircular of variance $(1+1/N)$ for all
$N$). Thus one has a canonical infinitesimal law associated to $X_{N}$.
Indeed, every moment of $X_{N}$ admits an expansion in powers of
$1/N$. Thus moments taken to order $t=1/N$ give rise to an infinitesimal
law, and the main reult of this section (Corollary \ref{cor:infLawXN})
states that the infinitesimal law associated to the matrices $X_{N}$
as $N\to\infty$ is the same as the infinitesimal law associated to
a type $B$ semcircular variable. 

\subsection{Matrices with entries free creation operators.}

Let $\ell(i,j,k)$ be a family of $*$-free creation operators in
a non-commutative probability space $(A,\psi)$. We thus assume that
\[
\ell(i,j,k)^{*}\ell(i',j',k')=\delta_{i=i',\ j=j',\ k=k'}1\]
and $\psi(w)=0$ whenever $w$ is a word involving $\ell(i,j,k)$'s
and their adjoints that cannot be reduced to a scalar using the relation
above. It is well-known (see e.g. {[}...]) that these two requirements
completely determine the joint $*$-distribution of this family and
moreover that the operators $\{\ell(i,j,k)\}_{i,j,k}$ are $*$-free.

Consider algebra $M_{N\times N}(A)=M_{N\times N}(\mathbb{C})\otimes A$
of $N\times N$ matrices with entries from $A$, endowed with the
state $\psi_{N}=\frac{1}{N}Tr\otimes\psi$. Let $L_{N}(k)$ be the
matrix\[
L_{N}(k)=\frac{1}{\sqrt{N}}(\ell(i,j,k))_{1\leq i,j\leq N}\]
and let $L_{N}(k)^{t}$ be its {}``transpose'':\[
L_{N}(k)^{t}=\frac{1}{\sqrt{N}}(\ell(j,i,k))_{1\leq i,j,\leq N}.\]

\begin{lemma}
The elements $L_{N}(k)$, $L_{N}(k)^{t}$ satisfy:\begin{eqnarray*}
L_{N}(k)^{*}L_{N}(k') & = & \delta_{k=k'}1\\
(L_{N}(k)^{t})^{*}L_{N}(k')^{t} & = & \delta_{k=k'}1\\
(L_{N}(k)^{t})^{*}L_{N}(k') & = & \delta_{k=k'}\frac{1}{N}\\
(L_{N}(k))^{*}L_{N}(k')^{t} & = & \delta_{k=k'}\frac{1}{N}.\end{eqnarray*}
Moreover, if $w$ is a word in $\{L_{N}(k),L_{N}(k)^{t}:k=1,2,\dots\}$
and their adjoint, which cannot be reduced to a scalar using these
relations, then $\psi_{N}(w)=0$. 
\end{lemma}
\begin{proof}
We compute the $a,b$-th entry of each of the matrices listed above
and use the relations among $\ell(i,j,k)$ and their adjoints. For
the matrix $(L_{N}(k)^{t})^{*}L_{N}(k')$, we get as the $a,b$-th
entry\[
\frac{1}{N}\sum_{i}\ell(a,i,k)^{*}\ell(i,b,k')=\delta_{k=k'}\delta_{a=b}\frac{1}{N}.\]
The other comutations are similar.

Finally, if $w$ is an irreducible words, then all entries of $w$
are either zero, or are irreducible words and thus yield zero when
supplied as an argument to $\psi$.
\end{proof}

\subsection{The law of $X_{N}$. }

\begin{corollary}
\label{cor:lawofXNis}Let $\{s(i,j,k):1\leq i\leq j\leq N,k=1,2,\dots\}$
be a free semicircular family, and let $X_{N}$ be an $N\times N$
matrix whose $i,j$-th entry is $(1+\delta_{i=j})N^{-1/2}s(i,j,k)$
if $i\leq j$ and $s(j,i,k)$ if $i>j$. Then $\{X_{N}(k):k=1,2,\dots\}$
form a free semicircular family, and each $X_{N}(k)$ is semicircular
with variance $(1+1/N)$. In particular, for each $k$, $\psi_{N}(X_{N}(k)^{m})=0$
if $m$ is odd, and\[
\psi_{N}(X_{N}(k)^{2n})=C_{n}\left(1+\frac{1}{N}\right)^{n},\]
where $C_{n}$ is the number of non-crossing pairings of $\{1,\dots,2n\}$. 
\end{corollary}
\begin{proof}
Let $\hat{X}_{N}(k)=\frac{1}{\sqrt{2}}(L_{N}(k)+L_{N}(k)^{*}+L_{N}(k)^{t}+(L_{N}(k)^{t})^{*})$.
Then $\hat{X}_{N}(k)$ is self-adjoint and its $i,j$-th entry, for
$i\leq j$ is equal to $ $$N^{-1/2}(\ell(i,j,k)+\ell(j,i,k)+\ell(i,j,k)^{*}+\ell(j,i,k)^{*})/\sqrt{2}$.
Since $l(i,j,k)=(\ell(i,j,k)+\ell(j,i,k))/\sqrt{2}$ is a free creation
operator for $i\neq j$ and is $\sqrt{2}$ times a free creation operator
for $i=j$, it follows that the joint law of the entries of $\{\hat{X}_{N}(k):k=1,2,\dots\}$
and $\{X_{N}(k):k=1,2,\dots\}$ are the same. Thus the laws of these
matrices are the same as well.

$L_{k}=2^{-\frac{1}{2}}(L_{N}(k)+L_{N}(k)^{t})$, so that $X_{N}(k)=L_{k}+L_{k}^{*}$.
Note that $L_{k}^{*}L_{k'}=(1+1/N)\delta_{k=k'}$ and $\psi_{N}(w)=0$
whenever $w$ is a word in $\{L_{k},L_{k}^{*}:k=1,2,\dots\}$ which
cannot be reduced to a scalar using this relation. It then follows
that each $L_{k}+L_{k}^{*}$ is semicircular of variance $(1+1/N)$,
and that $\{L_{k}:k=1,2,\dots\}$ are $*$-free. 
\end{proof}

\subsection{Infinitesimal laws associated to $X_{N}$ as $N\to\infty$. }

Let $\{X_{N}(k):k=1,2,\dots\}$ be matrices as defined in Corollary 
\ref{cor:lawofXNis}.
Thus the $i,j$-th entry $s(i,j,k)$ of $X_{N}(k)$ is a semicircular
variable of variance $N^{-1/2}(1+\delta_{i=j})$ and the variables
$\{s(i,j,k):1\leq i,j\leq N,k=1,2,\dots\}$ are assumed to be freely
independent.

We now consider the infinitesimal law of the family $\{X_{N}(k):k=1,2,\dots\}$
as $t=1/N$ approaches $0$ (so $N\to\infty$). Thus we define $\mu,\mu'$
by\begin{eqnarray*}
\mu(p) & = & \lim_{1/N\to0}\psi_{N}(p(X_{N}(1),\dots,X_{N}(m)),\\
\mu'(p) & = & \lim_{1/N\to0}\frac{1}{1/N}(\psi_{N}(p(X_{N}(1),\ldots,X_{N}(m))-\mu(p)),\end{eqnarray*}
for $p$ an arbitrary non-commutative polynomial in $m$ variables
$t_{1},\dots,t_{m}$.

\begin{corollary}
\label{cor:infLawXN}Let $(\mu,\mu')$ be as defined in the previous
paragraph. Then the variables $t_{1},\dots,t_{m}$ are infinitesimally
free when considered with the law $(\mu,\mu')$. Furthermore, the
restriction of $(\mu,\mu')$ to any variable $t_{k}$ yields an infinitesimal
(i.e., type $B$) semicircular law.
\end{corollary}
\begin{proof}
Indeed, the joint law of $\{X_{N}(k):k=1,2,\dots\}$ with respect
to $\psi_{N}$ is that of a free semcircular family of covariance
$(1+1/N)$. The statement of the corollary now follows using Corollary
\ref{cor29} applied to the case of the semicircular distribution and
from Proposition \ref{pro:B-free-implies-inffree}.
\end{proof}

As we noted in the introduction, it is necessary to start with a matrix $X_{N}$ with semicircular
entries; a real Gaussian random matrix will not work.

\section{Multiplicative free convolution of type $B$}

It has been shown in \cite{MVP} that one can define
a natural multiplicative analogue of the operation
$\boxplus_B$, which we will denote by $\boxtimes_B$.
Surprisingly, the main additive results have a formal multiplicative
analogue, with the notable exception of Theorem \ref{ans}. Unlike in the additive case,
it is far from clear what would be the narrowest appropriate class of analytic objects
for the second coordinate in the multiplicative case. Since
proofs are identical to the ones from the previous section, we will
provide only the statements of our results.

First we shall introduce several notations. For any probability
measure $\mu$, let 
$$\psi_\mu(z)=\int\frac{zt}{1-zt}\,d\mu(t),$$
be its {\em moment generating function.} If $\mu$ is supported on 
the unit circle $\mathbb T$ in the complex plane, then $\psi_\mu$
is defined and analytic inside the unit disc $\mathbb D$ and takes
values in the half-plane $\{z\colon\Re z\ge-1/2\}$. If $\mu$ is supported
on the positive half-line $[0,+\infty)$, then $\psi_\mu$ is an analytic
self-map of $\mathbb C\setminus[0,+\infty)$ which preserves the upper 
and lower half-planes and increases the argument: $\pi>\arg\psi_\mu(z)\ge
\arg z$ for $0<\arg z<\pi.$ In the following we shall denote by 
$\mathcal M_\mathbb T$ the set of Borel probability measures supported on the unit circle and by $\mathcal M_+$ the set of Borel probability 
measures supported on the positive half-line. 
Biane \cite{mathz} has shown that
a subordination result holds also for free multiplicative convolution:
\begin{theorem}\label{subordx}
Let $\mu_1,\mu_2$ be two Borel probability measures, and denote by
$\mu_3=\mu_1\boxtimes\mu_2$ their free multiplicative convolution.
\begin{enumerate}
\item If $\mu_1,\mu_2\in\mathcal M_\mathbb T$, then there exist
two analytic self-maps $\omega_1,\omega_2$ of the unit disc so that
\begin{trivlist}
\item[{\rm(a)}] $|\omega_j(z)|\leq|z|$, $z\in\mathbb D$, $j\in\{1,2\}$;
\item[{\rm (b)}] $\psi_{\mu_j}(\omega_j(z))=\psi_{\mu_1\boxtimes\mu_2}
(z)$, $z\in\mathbb D$, $j\in\{1,2\}.$
\end{trivlist}
\item If $\mu_1,\mu_2\in\mathcal M_+$, then there exist
two analytic self-maps $\omega_1,\omega_2$ of the slit complex plane
$\mathbb C\setminus[0,+\infty)$ so that
\begin{trivlist}
\item[{\rm(a)}] $\pi>\arg\omega_j(z)\ge\arg z$, $z\in\mathbb C^+$, $j\in\{1,2\}$;
\item[{\rm (b)}] $\psi_{\mu_j}(\omega_j(z))=\psi_{\mu_1\boxtimes\mu_2}
(z)$, $z\in\mathbb C\setminus[0,+\infty)$, $j\in\{1,2\}.$
\end{trivlist}
Moreover, in both cases the subordination functions satisfy the
following relation:
\begin{equation}\label{15}
\frac{z\psi_{\mu_1\boxtimes\mu_2}(z)}{1+\psi_{\mu_1\boxtimes\mu_2}(z)}
=\omega_1(z)\omega_2(z),
\end{equation}
for $z$ in the domain of $\psi_{\mu_1\boxtimes\mu_2}$.
\end{enumerate}
\end{theorem}
Using this result and the work of Popa \cite{MVP}, one can prove
the following results:
\begin{prop}\label{prop21}
Consider two type $B$ random variables $(a_1,\xi_1),(a_2,\xi_2)$ which
are $B$-free and distributed according to $(\mu_1,\nu_1)$ and 
$(\mu_2,\nu_2)$, respectively. Assume all these distributions are
compactly supported. Denote by $(\mu_3,\nu_3)$ the distribution of
$(a_1,\xi_1)(a_2,\xi_2)=(a_1a_2,a_1\xi_2+\xi_1a_2)$. Then, with the 
notations from Theorem \ref{subordx}, we have
\begin{enumerate}
\item[(a)] $\mu_3=\mu_1\boxtimes\nu_1$;
\item[(b)] $\displaystyle\frac{\psi_{\nu_3}(z)}{z}=
\frac{\psi_{\nu_1}(\omega_1(z))}{\omega_1(z)}
\omega_1'(z)+\frac{\psi_{\nu_2}(\omega_2(z))}{\omega_2(z)}\omega_2'(z).
$
\end{enumerate}
\end{prop}
Observe that we have not specified the domain on which the
functions above are defined, or the analytic nature of $\nu_j$, at this
moment. While the domains of the functions involved are rather easy
to find (the unit disc for distributions supported on $\mathbb T$ and 
the slit complex plane for distributions on the positive half-line), unfortunately it is not clear what appropriate sets of distributions
are stable under $\boxtimes_B$. The following analogue of Theorem
\ref{inf} is the only exception we know. 

\begin{theorem}\label{infx}
Assume that the functions $\gamma_j\colon[0,1]\to\mathcal M_\epsilon$
($\epsilon\in\{\mathbb T,+\}$) are differentiable on $(0,1)$ and 
$\gamma_j'$ extend continuously to $[0,1]$, $j\in\{1,2\}$. Then
$(\gamma_1(t),\gamma_1'(t))\boxtimes_B(\gamma_2(t),\gamma_2'(t))=
(\gamma_1(t)\boxtimes\gamma_2(t),\frac{d}{dt}(\gamma_1(t)\boxtimes
\gamma_2(t))),$ for all $t\in[0,1]$.
\end{theorem}

\begin{proof}
The proof is similar to the one of Theorem
\ref{inf}. We derivate in the subordination formula:
$$\partial_t\psi_{\gamma_1(t)\boxtimes\gamma_2(t)}(z)=
\psi_{\gamma_j'(t)}(\omega_j^t(z))+\psi_{\gamma_j(t)}'(\omega_j^t
(z))\partial_t\omega_j^t(z),$$
for $j\in\{1,2\}$, with $\omega_j^t$ denoting the subordination
function provided by Theorem \ref{subordx}.
From \eqref{15}, together with the formula above, we get

$$
\frac{z\partial_t\psi_{\gamma_1(t)\boxtimes\gamma_2(t)}(z)}{(1+
\psi_{\gamma_1(t)\boxtimes\gamma_2(t)}(z) )^2}=
\frac{\partial_t\psi_{\gamma_1(t)\boxtimes\gamma_2(t)}(z)-\psi_{\gamma_1'(t)}
(\omega_1^t(z))}{\psi_{\gamma_1(t)}'(\omega_1^t(z))}\omega_2^t(z)+
\frac{\partial_t\psi_{\gamma_1(t)\boxtimes\gamma_2(t)}(z)-\psi_{\gamma_2'(t)}
(\omega_2^t(z))}{\psi_{\gamma_2(t)}'(\omega_2^t(z))}\omega_1^t(z).$$
Recalling that $\psi_{\gamma_1(t)\boxtimes\gamma_2(t)}'(z)=
\psi_{\gamma_j(t)}'(\omega_j^t(z))(\omega_j^t)'(z)$, we may amplify 
the two fractions in the right hand term of the equality above by 
$(\omega_1^t)'(z)$ and $(\omega_2^t)'(z)$, respectively, and multiply
by $\psi_{\gamma_1(t)\boxtimes\gamma_2(t)}'(z)$ in order to get, after
regrouping, 
$$z\psi_{\gamma_1(t)\boxtimes\gamma_2(t)}(z)
\left(\frac{\psi_{\gamma_1(t)\boxtimes\gamma_2(t)}(z)}{1+
\psi_{\gamma_1(t)\boxtimes\gamma_2(t)}(z)}\right)'=
\psi_{\gamma_1(t)\boxtimes\gamma_2(t)}(z)[\omega_1^t(z)(\omega_2^t)'(z)
+(\omega_1^t)'(z)\omega_2^t(z)]
$$
$$
-\psi_{\gamma_1'(t)}(\omega_1^t(z))(\omega_1^t)'(z)\omega_2^t(z)
-\psi_{\gamma_2'(t)}(\omega_2^t(z))(\omega_2^t)'(z)\omega_1^t(z)
.$$
Equation \eqref{15} implies that
$z\left(\frac{\psi_{\gamma_1(t)\boxtimes\gamma_2(t)}(z)}{1+
\psi_{\gamma_1(t)\boxtimes\gamma_2(t)}(z)}\right)'-
\omega_1^t(z)(\omega_2^t)'(z)
-(\omega_1^t)'(z)\omega_2^t(z)=-\frac{\psi_{\gamma_1(t)\boxtimes\gamma_2(t)}(z)}{1+
\psi_{\gamma_1(t)\boxtimes\gamma_2(t)}(z)}.$
Using this relation, mutliplying the equation above by $[\omega_1^t(z)
\omega_2^t(z)]^{-1}$ and applying again \eqref{15} provides
$$\frac{\partial_t\psi_{\gamma_1(t)\boxtimes\gamma_2(t)}(z)}{z}=
\frac{\psi_{\gamma_1'(t)}(\omega_1^t(z))}{\omega_1^t(z)}(\omega_1^t)'(z) +
\frac{\psi_{\gamma_2'(t)}(\omega_2^t(z))}{\omega_2^t(z)}(\omega_2^t)'(z).
$$
The above equality holds for $z\in\mathbb D$ if $\epsilon=\mathbb T$
and for $z\in\mathbb C\setminus[0,+\infty)$ if $\epsilon=+$.
The theorem follows now from Proposition \ref{prop21}.
\end{proof}



\begin{thebibliography}{10}


\bibitem{akhieser} Akhieser, N. I. {\em The classical moment problem and some related questions in analysis. } Hafner Publishing Co., New York 1965



\bibitem{proc} Belinschi, S. T. {\em C-free convolution for measures
with unbounded support.} {$C^*$-algebras in Sibiu,} 1--7. Theta 2008



\bibitem{MZ} Belinschi, S. T.; Bercovici, H. {\em Atoms and regularity 
for measures in a partially defined free convolution semigroup.} Math. 
Z., {\bf 248}, (2004), 665-674.

\bibitem{BVIUMJ} Bercovici, H.; Voiculescu, D. {\em Free convolution
of measures with unbounded support.} Indiana Univ. Math. J. Vol. 42, No
3 (1993) 733--773.



\bibitem{mathz} Biane, Philippe. {\em Processes with free increments.}
Math. Z. 227 (1998), no.1, 143-174.

\bibitem{BGN} Biane, Philippe; Goodman, Frederick; Nica, Alexandru. {\em 
Non-crossing cumulants of type $B$}. {Trans. Amer. Math. Soc.} 355 (2003), no. 6, 2263--2303

\bibitem{BLS} Bo\.zejko, Marek; Leinert, Michael; Speicher, Roland.
{\em Convolution and limit theorems for conditionally free random variables.} {Pacific J. Math.} 175
(1996), no.2, 357--388.

\bibitem{ultra} Carmichael, Richard D.; Kami\'{n}ski, Andrzej; Pilipovi\'{c}, Stevan. {\em Boundary values and convolution in ultradistribution spaces.} Series on Analysis, Applications and Computation, 1.
World Scientific Publishing Co. 2007.

\bibitem{Edelmann} Dumitriu, Ioana; Edelmann, Alan. {\em Global spectrum fluctuations for the $\beta$-Hermite and $\beta$-Laguerre ensemble via matrix models.} J. Math. Phys. {\bf 47} {(2006)}, no. 6, 063302, 36pp.

\bibitem{Duren} Duren, Peter L. {\em Theory of $H^p$ spaces.} Pure and Applied Mathematics,
Vol. 38, Academic Press, New York-London 1970.


\bibitem{Johanssen} Johansson, K. {\em On fluctuations of eigenvalues 
of random Hermitian matrices.} Duke Math. J., 91(1): 151-204, 1998.

\bibitem{krysztek} {Krystek, Anna Dorota}. {\em Infinite divisibility 
for the conditionally free convolution}. Infin. Dimens. Anal. Quantum
Probab. Relat. Top. {\bf 10} {(2007)}, no.4, 499--522.


\bibitem{Luszczki} Luszczki, Z; Ziele\'{z}ny, Z. {\em Distributionen 
der R\"{a}ume $D'_{L^p}$ als Randverteilungen Analytischer Funktionen.}
Colloq. Math. {Vol. VIII} {(1961)}, Fasc. 1, 125--131.

\bibitem{MS2ndOrder} Mingo, James; Speicher, Roland. {\em Second
order freeness and fluctuations of random matrices: I. Gaussian and
Wishart matrices and cyclic Fock spaces.} J. Funct. Anal., 235, 2006,
pp. 226--270.

\bibitem{NSmult} Nica, Alexandru; Speicher, Roland. {\em On the 
multiplication of free $N$-tuples of noncommutative random variables},
Amer. J. Math. {\bf 118} (1996), 799--837.

\bibitem{NSbook} Nica, Alexandru; Speicher, Roland. {\em Lectures on the combinatorics of free probability.} {Cambridge University Press}, {2006}.

\bibitem{Levy} Pata, Vittorino. {\em L\'{e}vy type characterization of 
stable laws for free random variables.} {Trans. Amer. Math. Soc.} {\bf
347}, No.7, (1995), 2457--2472.

\bibitem{MVP} Popa, Mihai V. {\em Limit theorems and $S$-transform
in non-commutative probability spaces of type $B$}. {Preprint} (2007) arXiv:0709.0011.

\bibitem{reiner} Reiner, V. {\em Non-crossing partitions for classical reflection groups.} {Discrete Math.} {177} {(1997)}, 195--222

\bibitem{schwartz} Schwartz, Laurent. {\em Th\'{e}orie des 
distributions II.} Paris 1951.


\bibitem{RS2} Speicher, Roland. {\em Combinatorial theory of the free product with amalgamation and operator-valued free probability theory.} {Memoirs of the Amer. Math. Soc.} {132} {(1998)}, x+88

\bibitem{DVJFA} Voiculescu, Dan. {\em Addition of certain 
non-commutative random variables}. {J. Funct. Anal.}, 66 {(1986)}, 
323--346

\bibitem{VDN} Voiculescu, D.; Dykema, K.; Nica, A. {\em Free random
variables.} CRM Monograph Series, Vol. 1, AMS, 1992.


\bibitem{fe1} Voiculescu, Dan. {\em The analogues of entropy and of Fisher's information measure in free probability. I.}  Comm. Math. Phys.  155  (1993),  no. 1, 71--92. 

\bibitem{wang} Wang, Jiun-Chau. {\em Work in progress.}

\end{thebibliography}
\end{document}